\documentclass[a4paper,12pt]{article}

\usepackage{graphics,graphicx,amssymb,amsmath,verbatim,psfrag}
\usepackage{enumerate}
\usepackage{fancyhdr}
\usepackage{color}
\usepackage{mathrsfs}

\graphicspath{{figures/}}

\newcommand{\NN}{\mathbb{N}}

\newcommand{\RR}{\mathbb{R}}

\newtheorem{theorem}{Theorem}

\newtheorem{lemma}{Lemma}
\newtheorem{assumption}{Assumption}

\newtheorem{defi}{Definition}
\newtheorem{proposition}{Proposition}
\newtheorem{proc}{Procedure}
\newtheorem{problem}{Problem}
\newenvironment{proof}{\noindent {\em Proof.}}{\hfill \hspace*{1pt} \hfill $\square$}

\newtheorem{remm}{Remark}
\newenvironment{remark}{\begin{remm}\rm }{\hfill \hspace*{1pt} \hfill
$\circ$\end{remm}}

\newtheorem{exx}{Example}

\newcommand\diag{{\rm diag}}
\newcommand\real{\ensuremath{{\mathbb R}}}
\newcommand\complex{\ensuremath{{\mathbb C}}}

\newcommand\trace{{\rm Tr}}

\newcommand{\smallmat}[1]{\left[ \begin{smallmatrix}#1
  \end{smallmatrix} \right]}

\newcommand\mymatrix[2]{\left[\begin{array}{#1} #2 \end{array}\right]}

\allowdisplaybreaks[4]

\title{\bf Design of Marx generators as a structured eigenvalue assignment}

\begin{document}

\author{Sergio Galeani$^1$, Didier Henrion$^{2,3}$, Alain Jacquemard$^{4,5}$, Luca Zaccarian$^{2,6}$}

\footnotetext[1]{DICII, Universit\`a di Roma Tor Vergata, Italy.}
\footnotetext[2]{CNRS, LAAS, University of Toulouse, France.}
\footnotetext[3]{Faculty of Electrical Engineering, Czech Technical University in Prague,
Czech Republic.}
\footnotetext[4]{CNRS, IMB, Universit\'e de Bourgogne, Dijon, France.}
\footnotetext[5]{Wolfgang Pauli Institut, Vienna, Austria.}
\footnotetext[6]{Dipartimento di Ingegneria Industriale, University of Trento, Italy}

\maketitle

\begin{abstract}
We consider the design problem for a Marx generator electrical
network, a pulsed power generator.
The engineering specification of the design is that a suitable
resonance condition is satisfied by the circuit so that the energy
initially stored in a number of storage capacitors is transferred in
finite time to a single load capacitor which can then store the total
energy and deliver the pulse.
 We show that the components design can be
conveniently cast as a structured real eigenvalue assignment
with significantly lower dimension than the state size of the Marx
circuit. Then we comment on the nontrivial nature of this structured
real eigenvalue assignment problem and present two possible approaches
to determine its solutions. A first symbolic approach consists in the use of
Gr\"obner basis representations, which allows us to compute all the
(finitely many) solutions. A second approach is based on
convexification of a nonconvex optimization problem with polynomial
constraints. We show that the symbolic method easily provides
solutions for networks up to six stages while the numerical method can
reach up to seven and eight stages. We also comment on the conjecture
that for any number of stages the problem has finitely many solutions,
which is a necessary assumption for the proposed methods to
converge. We regard the proof of this conjecture as an interesting
challenge of general interest in the real algebraic geometry field.
\end{abstract}

\section{Introduction}
\label{sec:intro}

Electrical pulsed power generators have been studied from the 1920s
with the goal to provide high power electrical pulses by way of
suitable electrical schemes that are slowly charged and then,
typically by the action of switches, are rapidly discharged to provide
a high voltage impulse or a spark (see, e.g., \cite{Bluhm06}).
Marx generators (see \cite[\S 3.2]{Bluhm06} or \cite{Carey02} for an
overview) were originally described by E. Marx in 1924 and correspond
to circuits enabling generation of high voltage from lower voltage
sources.
While many schemes have been proposed over the years for Marx
generators, a recent understanding of certain compact Marx generators
structures \cite{BuchenauerTPS10} reveals
that their essential behavior can be well described by a suitable $LC$
ladder network where certain components should be designed in order to
guarantee a suitable resonance condition. In turn, such a resonance
condition is known to lead to a desirable energy transfer throughout
the circuit and effective voltage multiplication which can then be
used for pulsed power generation.

\begin{figure}[ht!]
\begin{center}
\includegraphics[width=\columnwidth]{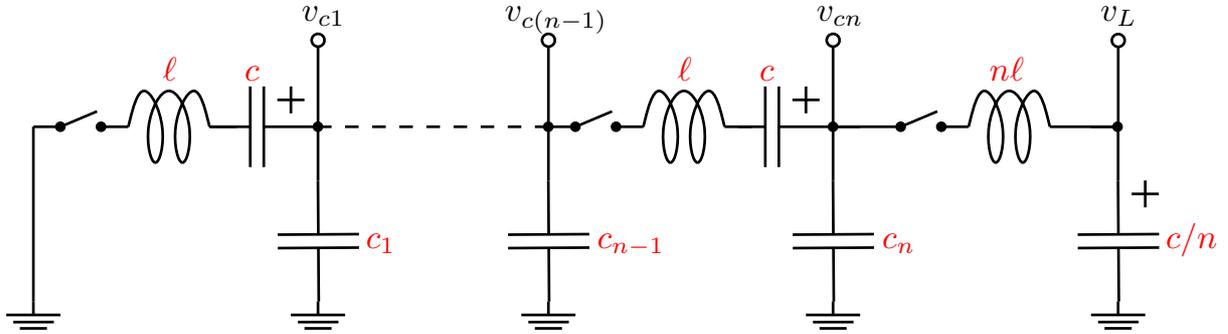}
\end{center}
\caption{The passive circuit used as a Marx generator.}
\label{fig:circuit}
\end{figure}

{\color{black}
This paper addresses the mathematical problem of designing the lumped
components of the compact Marx generator circuit well described in 
\cite{BuchenauerTPS10} and represented in Figure~\ref{fig:circuit}.
In particular, in \cite{BuchenauerTPS10},
based on a Laplace domain representation of the ladder network of Figure~\ref{fig:circuit},
 several experimentally
oriented discussions are provided illustrating that, as long as the
vertical (parasitic) capacitors of the network are
suitably selected to guarantee a certain resonance property, 
the network performs very desirably in terms of pulsed power
generation. Despite its fundamental importance from the experimental
viewpoint,  \cite{BuchenauerTPS10} does not provide a viable technique
for ensuring this resonance property and uses heuristics to find some
solutions to the design problem for some fixed number of
stages. A similar approach has been taken in the Master thesis
\cite{Antoun06}, where a state-space description of the same circuit
has been presented and a nonlinear least squares approach has been
proposed to ensure the desirable resonance property. 
Both the works outlined above have followed heuristic approaches
for the computation of the circuit components ensuring
resonance. Conversely, in this paper we introduce a substantially
different approach to the design problem by showing, via suitable 
transformations, that this design
can be cast as a 
%
%In this paper we follow-up on the preliminary results of
%\cite{Antoun06,ZackPPC09}, which are strongly motivated by the
% observations in \cite{BuchenauerTPS10}, and we transform a compact
% Marx generator design problem into a 
structured real eigenvalue
 assignment problem for a linear system associated with the Marx
 network of Figure~\ref{fig:circuit} and only depending on the number of its stages.
}

The problem of static output feedback pole (or eigenvalue) assignment for linear systems has been
largely studied in the 1990s, see \cite{RosWil99} for a survey.
In its simplest form, it can be stated
as follows: given matrices $A \in {\mathbb R}^{n\times n}$, $B \in {\mathbb R}^{n\times m}$,
$C \in {\mathbb R}^{p\times n}$ and a monic polynomial $q(s) \in {\mathbb R}[s]$
of degree $n$, find an $m\times p$ real matrix $F$ such that $\det(sI_n-A-BFC)=q(s)$
where $I_n$ denotes the identity matrix of size $n$.
This problem has generically a solution if $mp>n$, and it has generically no solution
if $mp<n$.
 The situation $mp=n$ is much more subtle.
For this situation, it was proved in \cite{EreGab02a,EreGab02b}
that the pole placement map from the feedback matrix $F$ to the characteristic polynomial $q(s)$ is generically
not surjective. It means that there is a non-empty open subset of real matrices $A$,$B$,$C$
for which there exist open sets of pole configurations symmetric w.r.t. the real axis
which cannot be assigned by any real feedback. 
In this paper, we do not consider static output feedback
pole assignment in the form just described, but in a
structured form. The number of degrees of freedom
(number of entries in the feedback matrix) is
equal to $n$, the number of poles to be assigned,
so it bears similarity with the difficult case
$mp=n$ of static output feedback pole assignment
described above.

%(and corresponds to the problem characterized in this paper). For this
%situation, it was proved in \cite{EreGab02a,EreGab02b}
%that the pole placement map from the feedback matrix $F$ to the characteristic polynomial $q(s)$ is generically
%not surjective. It means that there is a non-empty open subset of real matrices $A$,$B$,$C$
%for which there exist open sets of pole configurations symmetric w.r.t. the real axis
%which cannot be assigned by any real feedback. 
%
Within the context of the above cited works \cite{RosWil99,EreGab02a,EreGab02b},
the structured pole assignment problem
that we characterize in this paper corresponds to the case where $C=I_n$ and $F$ is diagonal,
which is a case not covered by these works.
For this problem, we provide in this paper a few general results characterizing its
solutions, and then show how one can tackle the solution. 
For
computationally tractable cases,
a first
technique, based on Gr\"obner basis representations, 
allows determining the full set of solutions, or in other words all the
possible selections of the circuit components leading to the desired
energy transfer.
A second numerical technique, based on nonconvex polynomial
optimization tools, allows determining one solution which is optimal
in the sense that it minimizes a certain polynomial cost function.
In the paper we state assumptions requiring that the solution set is
nonempty (that is, there exists at least one choice of the parameters
leading to the desired resonance condition) and that the number of
solutions is actually finite, namely the set of all solutions is a
zero-dimensional set. These assumptions imply desirable
termination and convergence properties of the symbolic and numerical
techniques which are highlighted in the paper. Interestingly,
numerical evidence reveals that these assumptions are satisfied for
all the numerically tractable cases that have been considered, however
their general validity is an interesting open problem.
Additional work related to the techniques reported here corresponds to
 \cite{Chi05} 
and references therein, where algebraic techniques (elimination
theory using resultants) are used in the design of power electronic devices.
In this reference the authors show that parasitic higher order harmonics in
a multilevel converter can be removed by solving a structured polynomial
system of equations featuring a high degree of symmetry.

{\color{black}
Preliminary results along the direction of this paper were presented
in \cite{ZackPPC09}, where Laplace domain descriptions were used and
where a result roughly 
 corresponding to one of the key components of our proof (that is,
Lemma~\ref{lem:Wpoles} in Section~\ref{sec:reson})
 was stated without any proof.
 Here, in addition
to proving that result, we provide a more complete statement
establishing sufficient conditions for the desired pulsed power generation.
Moreover, we
adopt a state-space representation that allows to provide a very
elegant and compact proof of our main result. 
Finally, an important
contribution of this paper consists in the two solution methods
outlined above. The symbolic one allows us to enumerate all the possible components
selections ensuring the resonant conditions for circuits with up to
$n=6$ stages. The numerical one allows us to compute the so-called
``regular solution'' (characterized in Section~\ref{sec:alg_sens}) for
the more convoluted cases $n=7$ and $n=8$. 
{\color{black}
While typical experimental realizations of the Marx generator of
Figure~\ref{fig:circuit} does not involve a larger number of stages,
the problem of solving the presented structured eigevnalue problem for
larger values of $n$ constitutes an interesting benchmark problem for 
researchers in the algebraic geometric field.
}

The contributions of this paper consist in the
following three points: A) establishing
sufficient conditions for the energy transfer goal behind the
architecture of Figure~\ref{fig:circuit} (the problem statement and
main results are presented in Section~\ref{sec:design} and the proofs
are reported in Section~\ref{sec:proofs}). B) illustrating two methods
to solve those conditions: a symbolic one, treated in
Section~\ref{sec:symbolic}, and a
numerical one, treated in Section~\ref{sec:numerical}. 
Both methods are of interest for the addressed experimental
problem due to suitable trade-offs between complexity and achievable
results;
C) illustration of the 
 potential behind the adopted Gr\"obner basis representation and the
 adopted numerical optimization of convex relaxations by way of suitable
 ``engineering-style'' propositions establishing the applicability
 limits of thsoe two approaches.
Finally, some discussions about the above mentioned conjectures and the
relevance of this study as a challenge within the algebraic geometric
field are given in
Section~\ref{sec:conclusions}, while Appendix~\ref{sec:maple} contains some illustrative
smaple Maple code implementing our algorithms.
}

%The paper is structured as follows: in Section~\ref{sec:design} we
%described the Marx generator network, we formalize the design problem
%and we provide the main theorem of the paper, transforming the design
%into an eigenvalue assignment problem.
%In Section~\ref{sec:symbolic} we present the symbolic solution 
%to the eigenvalue assignment problem based
%on Gr\"obner basis representations. In Section~\ref{sec:numerical} we
%present the numerical solution based on convexified formulations. The 
% proof of the main theorem is
%reported in Section~\ref{sec:proofs} at the end of the paper for a
%simplified reading. Finally, Section~\ref{sec:conclusions} contains
%some discussions about the above mentioned conjectures and the
%relevance of this study as a challenge within the algebraic geometric field.

{\bf Notation}: Given a square matrix $A$, $\sigma(A)$ denotes its
spectrum, i.e., the set of its complex eigenvalues.
Given a vector $f \in \real^n$, ${\mathbb Q}[f]$ denotes the set of
all polynomials with rational coefficients in the indeterminates $f$. 

%\section{Marx generator design as an inverse eigenvalue problem}
\section{Marx generator design}
\label{sec:design}

\subsection{Circuit description and problem statement}

We consider the Marx generator network shown in
Figure~\ref{fig:circuit} consisting in $n$ stages (and $n+1$ loops) where,
disregarding the two rightmost components of the figure,
each one of the
$n$ stages
consists in 1) an upper branch with a capacitor and an
inductor and 2) a vertical branch with a capacitor only. Following
\cite{Antoun06,ZackPPC09,BuchenauerTPS10}, we assume that all the
capacitors and
inductors appearing in the upper branches are the same (corresponding
to some fixed positive reals $c$ and
$\ell$). We will call these capacitors ``storage capacitors'' in the
sequel, for reasons that will become clear next.
The design problem addressed here is the selection of the vertical
capacitors, which are exactly $n$, where $n$ is the number of stages
of the Marx circuit. We will call these capacitors ``parasitic
capacitors'' due to their position resembling that of parasitic
capacitors in transmission lines. Despite their name, the parasitic
capacitors $c_i$, $i=1,\ldots, n$ are not
necessarily arising from any parasitic effects and their values will
be selected in such a way to ensure a suitable resonance condition as
clarified next.

Following \cite{Antoun06,ZackPPC09,BuchenauerTPS10}, the inductance
and capacitor appearing in the rightmost loop take the values $n \ell$ and
$c/n$, respectively. We call this capacitor the ``load capacitor''.
This selection preserves the resonance property
(so that the product of any adjacent capacitor/inductor pairs is
always $\ell c$) in addition to ensuring that the load capacitor is
$n$ times larger than each one of the storage capacitors. The
problem addressed in this paper (resembling that one tackled in
\cite{Antoun06,ZackPPC09,BuchenauerTPS10}) is the following.

\begin{problem}
\label{prob_stat}
Consider the circuit in Figure~\ref{fig:circuit} for given $n$ and
certain values of $c$ and $\ell$. Select positive values $c_i>0$,
$i=1,\ldots n$ of the
parasitic capacitors and a time $T>0$ such that, initializing at $t=0$ all the storage
capacitors with the same voltage $v(0)=v_{\circ}$ and starting from zero current
in all the inductors
and zero voltage across the parasitic capacitors and the load
capacitor, the
circuit response is such that at $t = T$ all voltages and currents are zero
except for the voltage across the load capacitor.
\end{problem}

\begin{remark}
The resonance condition required in Problem~\ref{prob_stat}
is a key feature to allow the use of the circuit of
Figure~\ref{fig:circuit} for pulsed power generation. In particular,
assuming that a number of switches are used to open the loops of the
circuit, the $n$ storage capacitors can be charged at the same voltage
$v_{\circ}$. Due to the property in  Problem~\ref{prob_stat}, at time
$T$,
 all the energy initially stored in the storage capacitors will be
 concentrated in the load capacitor. In particular,
since this is a lossless circuit,
 we will have
$$
\sum_{i=1}^n \frac{1}{2} c v_{\circ}^2 = \frac{1}{2} \frac{c}{n} v_L^2 (T),
$$
which clearly implies $v_L(T) = n v_{\circ}$, namely both the
voltage and
energy transferred to the load capacitor is $n$ times larger than the
voltage and energy initially stored in each one of the storage capacitors.
\end{remark}

\subsection{Solution via structured eigenvalue assignment}

In this section we show that a solution 
 to Problem~\ref{prob_stat} can be determined from the solution of a suitable
 structured eigenvalue assignment problem involving a matrix
$B\in \real^{n \times n}$ defined as
\begin{align}
\label{eq:matrix_B}
B:=
\smallmat{
		2 & -1 & 0 & \cdots & 0\\
		-1 & \ddots & \ddots & \ddots & \vdots\\
		0 & \ddots & \ddots & \ddots & 0\\
		\vdots & \ddots & \ddots & 2 & -1 \\
		0 & \cdots & 0 & -1 &  \frac{n + 1}{n}
}
\end{align}
and an arbitrary set of even harmonics of the fundamental frequency
$\omega_0 = \sqrt{(\ell c)^{-1}}$ to be assigned to the circuit.
The following is the main result of this paper, whose proof is given
in Section~\ref{sec:proofs}.

\begin{theorem}
\label{th:design}
Consider any set of $n$ distinct positive even integers $\alpha = (\alpha_1,
\ldots, \alpha_n)$,
the matrix $B$ in (\ref{eq:matrix_B}) and any positive definite real
diagonal solution $F = \diag(f_1,\ldots,f_n)$
to the structured eigenvalue assignment problem
\begin{align}
\label{eq:assignBF}
\sigma (BF) = \{ \alpha_1^2-1, \ldots, \alpha_n^2-1\}.
\end{align}
Then for any value of $c$, the selection $c_i = c/f_i$, $i=1, \ldots,
n$, solves Problem~\ref{prob_stat} for all values of $\ell$ with
$T=\frac{\pi}{\sqrt{\ell c}}$.
\end{theorem}

Theorem~\ref{th:design} shows that a solution to
Problem~\ref{prob_stat} can be determined by solving an eigenvalue assignment
problem with decentralized feedback (because matrix $F$ is diagonal).
{\color{black} 
Note that this structured eigenvalue assignment problem arises
naturally from the physical nature of the circuit under consideration
and does not arise from some simplifying assumptions on the circuit behavior.
}
A generalized version of this structured pole assignment problem
was studied in \cite{Wan94}
(indeed, using the notations there, the pole assignment problem 
in (\ref{eq:assignBF}) is obtained by setting $r=n$, $C=I_n$, $m_i=p_i=1$, $i=1,\ldots,n$).
It is shown in \cite{Wan94} that generic pole assignment depends on the dimension of a product
Grassmannian, see \cite[Equ. (17)]{Wan94}. In our case it is equal to $n!$
which is always even, and from \cite[Theorem 4.2]{Wan94} it follows that
generic pole assignment cannot be achieved. A geometric condition
that ensures generic pole assignment is given in \cite[Prop. 4.2]{Wan94} but we
do not know whether this condition can be checked computationally.
In any case it is an evidence that the question of existence
of a real solution to our inverse eigenvalue problem does not appear
to be trivial.

%{\bf DIDIER: I suggest to put the following two remarks here,
%before the symbolic and numerical solution sections}

%\begin{remark}
%The inverse eigenvalue problem (\ref{eq:assignBF}) corresponds to a
%structured eigenvalue assignment where the matrix $B$ is fixed by
%(\ref{eq:matrix_B}) and the assignment is performed by selecting a
%structured diagonal state feedback matrix $F$.
%{\tt comment on other structured eig assignments -- email from Didier}
%\end{remark}

\begin{remark}
Bearing in mind that $F$ is diagonal, real and positive definite,
the inverse eigenvalue problem (\ref{eq:assignBF}) can be equivalently
cast as a symmetric inverse Jacobi eigenvalue problem (IJEP) (see,
e.g., \cite{ChuSIAM98,ChuBook05}) by performing the coordinate change with $T =
\sqrt{F}$, which leads to
\begin{align}
\label{eq:symm}
M := T (BF) T^{-1} = \sqrt{F} B \sqrt{F},
\end{align}
where it is evident that $M= M^T >0$. In particular, positive
definiteness of $M$ arises from positive definiteness of $B$ which
is established in the proof of Lemma~\ref{lem:Wpoles}.
Since (\ref{eq:symm}) is a change of coordinates, then
assigning the spectrum $\{ \alpha_1^2-1, \ldots, \alpha_n^2-1\}$
 to the matrix $M$ is equivalent to solving
the eigenvalue assignment problem (\ref{eq:assignBF}) with the possible advantage that
$M$ is a symmetric matrix.
\end{remark}

The following result shows that any solution to (\ref{eq:assignBF}) is
physically implementable as it corresponds to positive values of the
parasitic capacitors.

\begin{lemma}
\label{lem:Fpos}
Any solution to the structured eigenvalue assignment problem
(\ref{eq:assignBF}) in Theorem~\ref{th:design} is such that $F>0$.
\end{lemma}

\begin{proof}
Given any solution to (\ref{eq:assignBF}), all the diagonal entries of $F$ are
necessarily nonzero otherwise $BF$ would be rank deficient which
contradicts the fact that $\sigma (BF)$ only has strictly positive elements.
Define $F = \bar F D$,
where $\bar F>0$ is diagonal and $D$ is a diagonal matrix whose
elements are either $1$ or $-1$. 

Assume now that the statement of the proposition is not true
so that $D$ has at least one negative entry. We consider 
the coordinate transformation $\sqrt{\bar F} BF
\sqrt{\bar F}^{-1} = \sqrt{\bar F} B\sqrt{\bar F} D =: \bar M D$, where the last
equality follows from the fact that $\bar F$ and $D$ are both
diagonal. Due to the coordinate transformation, all the eigenvalues of
$\bar M D$ are positive (because they are the same as those of $BF$).
Since $B>0$ (see the proof of Lemma~\ref{lem:Wpoles}
in Section~\ref{sec:reson}) we have that $\bar M := \sqrt{\bar F}
B\sqrt{\bar F}$ is symmetric positive definite, so that there exists
an orthogonal matrix $Q$ such that $\bar M  = Q \Lambda Q^T$, with
$\Lambda$ diagonal positive definite. Then it follows that 
$Q^T \bar MD = \Lambda Q^T D$ which, pre-multiplied by
$(\sqrt{\Lambda})^{-1}$ and post-multiplied by $Q\sqrt{\Lambda}$ leads
to
$$
\Sigma := (\sqrt{\Lambda})^{-1}Q^T \bar MD Q\sqrt{\Lambda} = \sqrt{\Lambda} Q^T DQ\sqrt{\Lambda},
$$
which establishes that the matrix $\Sigma$ is symmetric (rightmost
term) and has positive eigenvalues (middle term which is a coordinate
transformation from $BF$). Nevertheless, the matrix $D$ is not
positive definite by assumption, which leads to a contradiction.
\end{proof}

\subsection{Two equivalent eigenvalue assignment formulations}
\label{sec:twoform}

Selecting the diagonal entries
$f = [f_1 \cdots f_n]^T$ to solve
the eigenvalue assignment (\ref{eq:assignBF}) amounts to solving a finite set
of $n$ equations (polynomial in the unknown $f$ with rational coefficients), each of them corresponding to one
coefficient of the following polynomial identity in the variable $s$:
\begin{equation}
\label{eq:pol_assign}
\det(s I - BF) = \prod_{i=1}^n \Big(s- (\alpha_i^2-1)\Big), \quad \forall s \in \complex .
\end{equation}                 
As an example, for the case $n=2$,
where $BF = \smallmat{2 f_1 & -f_1 \\ -f_2 & \frac{3}{2}f_2}$,
 if one makes the simple~\footnote{A typical selection of
   $\alpha$ is $\alpha_i = 2i$, so that the circuit resonates at the
   lowest possible frequency.}
 selection 
$\alpha = (2, 4)$, then the set of polynomial equations ensuring 
(\ref{eq:pol_assign}) corresponds to:
\begin{equation}
\label{eq:ex3}
\left\{
\begin{array}{rcl}
2 f_1  f_2 &=&   (\alpha_1^2-1)(\alpha_2^2-1)= 45  \\
2 f_1+ \frac{3}{2} f_2 &=& (\alpha_1^2-1) + (\alpha_2^2-1) = 18  .
\end{array}
\right.
\end{equation}
In the general case, for a fixed value of $n$ and fixed values in $\alpha$, one can
write a system of $n$ polynomial equations in the variable
$f$ with rational coefficients, namely
\begin{equation}
\label{eq:systemp}
p_i(f) = 0, \quad i=1,\ldots, n.
\end{equation}
In this paper, we will also adopt
an alternative formulation of the problem which appears to be more suitable for the
numerical optimization techniques developed in
Section~\ref{sec:numerical}.
% and will
%also be adopted in Section~\ref{sec:symbolic} to allow for a
%comparison of the symbolic and numerical solution methods.  
The alternative formulation corresponds to inverting the eigenvalue assignment problem
(\ref{eq:assignBF}), thereby obtaining an alternative set of polynomial equations in the unknowns
$k = [k_1 \cdots k_n]^T = [f_1^{-1} \cdots
f_n^{-1}]^T$ with rational coefficients, which have the advantage of 
being linear in the capacitor values, indeed, $k_i = c_i/c$, $i=1,\ldots ,\ n$.
In particular, for the inverse problem, equation (\ref{eq:pol_assign}) becomes
\begin{equation}
\label{eq:pol_assign_inv}
\det(s I - K B^{-1}) = \prod_{i=1}^n \Big( s- (\alpha_i^2-1)^{-1} \Big), \quad \forall s \in \complex ,
\end{equation} 
where $K=\diag(k) = F^{-1}$ is a diagonal matrix whose diagonal
elements are the scalars to be determined. 
Similar to above, for the case $n=2$, we have $KB^{-1} = \smallmat{\frac{3k_1}{4} & \frac{k_1}{2} \\
 \frac{k_2}{2} & k_2}$ which, for the simple selection $\alpha = (2, 4)$ leads to the
following set of polynomial equations ensuring
(\ref{eq:pol_assign_inv}) (and, equivalently, (\ref{eq:pol_assign})):
\begin{equation}
\label{eq:ex3_inv}
\left\{
\begin{array}{rcl}
\frac{1}{2} k_1 k_2 &=&   (\alpha_1^2-1)^{-1}(\alpha_2^2-1)^{-1}= \frac{1}{45}  \\
\frac{3}{4} k_1 + k_2 &=& (\alpha_1^2-1)^{-1} + (\alpha_2^2-1)^{-1} = \frac{2}{5}  .
\end{array}
\right.
\end{equation}
In the general case, for a fixed value of $n$ and fixed values in $\alpha$, one can
write a system of $n$ polynomial equations in the variable
$k$ with rational coefficients, namely
\begin{align}
\label{eq:systemq}
q_i(k) = 0, \quad i=1,\ldots, n.
\end{align}
Formulation (\ref{eq:pol_assign_inv}), (\ref{eq:systemq}) will be used
in Section~\ref{sec:numerical} due to the advantageous property that
all entries of each solution to this polynomial system are in the interval $(0,1)$ as
established in the next lemma.

\begin{lemma}
\label{lem:Kbox}
Given any $n\geq 1$ and
any set of distinct positive even integers $\alpha = (\alpha_1,
\ldots, \alpha_n)$, each solution $k = (k_1,\ldots,k_n)$ to the
inverse eigenvalue problem (\ref{eq:pol_assign_inv}) is such that 
$0< k_i < 1$ for all $i=1,\ldots,n$.
\end{lemma}

\begin{proof}
The fact that $k_i>0$ for all $i=1,\ldots,n$ follows from
Lemma~\ref{lem:Fpos} and $K=\diag(k_1,\ldots,k_n) = F^{-1}>0$.

Denote now by $d_i$, $i=1,\ldots,n$ the diagonal entries of $B^{-1}$
and let us show next that $d_i \geq \frac{1}{2}$ for all $n\geq 1$,
for all $i=1,\ldots, n$. To see this,
%first let us recall that, as proven in the proof of Lemma~\ref{lem:Wpoles}
% (Section~\ref{sec:reson}), $B$ is positive definite. Then,
 for each $i\in \{1,\ldots,n\}$, denote by $e_i$ the $i$-th unit
 vector of the canonical basis of the Euclidean space and
consider the coordinate change 
matrix $T=\smallmat{e_i & e_2 & \cdots & e_{i-1} & e_1 & e_{i+1} &
  \cdots & e_n}$, so that $T_i=T_i^T=T_i^{-1}$ and so that $T_iBT_i$
exchanges the $1$-st and the $i$-th 
rows and columns of $B$. Then it is readily seen that $T_iBT_i$ is
positive definite (by positive definiteness of $B$, as established in 
the proof of Lemma~\ref{lem:Wpoles} in Section~\ref{sec:reson}) and
that the $(1,1)$ element of $T_iB^{-1}T_i = (T_iBT_i)^{-1}$
 corresponds to $d_i$. Partition the matrix
as $T_iBT_i = \smallmat{b_i & b_{i,12}^T\\  b_{i,12} & B_{i,22}}$,
where $b_i$ is the $i$-the diagonal element of $B$ satisfying by
construction $b_i \leq 2$. Since $B_{i,22}>0$, 
the following holds by applying 
 the matrix inversion
formula to the (1,1) element of $T_iB^{-1}T_i = (T_iBT_i)^{-1}$:
\begin{equation}
\label{eq:good_di}
d_i = \left(b_i - b_{i,12}^T\ B_{i,22}^{-1}\  b_{i,12}\right)^{-1}
> b_i^{-1} \geq 2^{-1}.
\end{equation}
As a next step, considering that $\alpha_i$, $i=1,\ldots, n$ are
distinct positive even integers, we have that
\begin{align}
\nonumber
%\begin{array}{ll}
&\displaystyle \sum_{j=1}^n \frac{1}{\alpha_j^2- 1} 
\leq\displaystyle \sum_{j=1}^n\frac{1}{(2j)^2- 1} 
=\frac{1}{2}\displaystyle \sum_{j=1}^n\frac{1}{2j-1}+\frac{1}{2j+1}\\
\nonumber
&=\displaystyle \frac{1}{2}\left[\left(1-\frac{1}{3}\right)+\left(\frac{1}{3}-\frac{1}{5}\right)
+\cdots+\left(\frac{1}{2n-1}-\frac{1}{2n+1}\right)\right]\\
&=\displaystyle \frac{1}{2}\left(1-\frac{1}{2n+1}\right)
\leq \frac{1}{2}.
\label{eq:good_alpha}
%\end{array}
\end{align}
%
%\begin{equation}
%\label{eq:good_alpha}
%\sum_{j=1}^n \frac{1}{\alpha_j^2- 1} \leq \sum_{j=1}^n
%\frac{1}{(2j)^2- 1} =\frac{n}{2n+1} = \frac{1}{2}-\frac{1}{2(2n+1)}
%\leq \frac{1}{2}.
%\end{equation}
%
Finally, keeping in mind that the trace of $KB^{-1}$ is the sum of its
eigenvalues and considering equality (\ref{eq:pol_assign_inv}) we can
combine (\ref{eq:good_di}) and (\ref{eq:good_alpha}) to get
$$
\frac{1}{2} \sum_{j=1}^n k_j < \sum_{j=1}^n d_j k_j = \trace
(KB^{-1}) = \sum_{j=1}^n \frac{1}{\alpha_j^2- 1}\leq \frac{1}{2},
$$
which, bearing in mind that $k_j>0$, for all $j\in \{1,\ldots, n\}$, 
implies, for each $i\in \{1,\ldots, n\}$, $k_i \leq \sum_{j=1}^n k_j
<1$ as to be proven.
\end{proof}

\begin{remark}
\label{rem:sensitivity}
(Sensitivity analysis)
%{\tt we need to eventually be more specific on this and use these
%  techniques to assess sensitivity of our solutions}
Sensitivity of the solution obtained by numerical techniques
(or also by symbolic techniques, by truncating rational numbers
to floating point numbers) can be assessed a posteriori.
Indeed, the eigenvectors of a matrix encode the sensitivity
of the eigenvalues to (unstructured) uncertainty affecting the
entries of the matrix itself. In particular, sensitivity can be and
will be assessed numerically in Section~\ref{sec:alg_sens} using two
main methods: a first one providing only local information
by using the condition number of simple
eigenvalues given in \cite[\S 7.2.2]{Golub96} and a second one
providing an idea about the effect of large perturbations by
graphically displaying 
%
%(open-loop) matrix $A$ (see e.g. the chapter on eigenvalues
%and singular values in \cite{Mol04}).
%Alternatively, we can use a graphical approach by displaying
% 
the pseudo-spectrum of the matrix (namely the sublevel sets of the
perturbed eigenvalues under norm bounded unstructured perturbations) 
according to \cite{TreEmb05}. As shown in Section~\ref{sec:alg_sens},
these tools provide useful insight about the features of different
solutions leading to interesting interpretations of some of the
observations in \cite{BuchenauerTPS10}.
\end{remark}

\section{Symbolic solution}
\label{sec:symbolic}

In this section we use techniques
from real algebraic geometry to solve
symbolically the
inverse eigenvalue problem (\ref{eq:assignBF}).
We focus on system \eqref{eq:systemp} (issued from \eqref{eq:pol_assign})
whose polynomials $p_i$ conveniently inherit some sparsity
of the tridiagonal matrix $B$, in the sense that
many monomials of the $f_i$ variables are zero. 
In contrast, the polynomials $q_i$ in system
\eqref{eq:systemq} (issued from \eqref{eq:pol_assign_inv}) are less sparse since matrix
$B^{-1}$ is dense.

%\subsection{Assumptions}
{\color{black}
Solving problem \eqref{eq:systemp}
amounts to finding a real $n$-dimensional solution
$f=(f_1,\ldots,f_n)$
to a system of $n$ given scalar-valued multivariate 
polynomial equations with rational coefficients.
%
%polynomial equations with rational coefficients. Let us denote
%this system as
%\begin{equation}\label{eq:polsys}
%\begin{array}{rcl}
%p_1(f_1,\ldots,f_n) & = & 0 \\
%p_2(f_1,\ldots,f_n) & = & 0 \\
%& \vdots & \\
%p_n(f_1,\ldots,f_n) & = & 0
%\end{array}
%\end{equation}
%where $p_i$ are given scalar-valued multivariate polynomials
%and $f_i$ are scalars to be found, $i=1,\ldots,n$.
%
Throughout the paper we make the next standing assumptions on
such a system of equations.
}
%\begin{assumption}\label{as:rational}
%The coefficients of the system of polynomial equations
%are rational numbers.
%\end{assumption}

%Since $B$ is a matrix with rational entries,
%Assumption~\ref{as:rational} is satisfied if each $\alpha^2_i$,
%$i=1,\ldots,n$ in \eqref{eq:pol_assign} is a rational number.

\begin{assumption}\label{as:finitesol}
There exists a finite number of complex
solutions to the system of polynomial equations (\ref{eq:systemp}).
\end{assumption}

Note that Assumption~\ref{as:finitesol}
readily implies that there is a finite number
of real solutions to the system of polynomial equations
(\ref{eq:systemp}).

\begin{assumption}\label{as:onesol}
There exists at least one real solution to the system of
polynomial equations (\ref{eq:systemp}).
\end{assumption}

In the conclusions section we state open problems
in connection with these assumptions.

\subsection{Real algebraic geometry}

For an elementary tutorial account of real algebraic geometry,
please refer to \cite{cox}. In this paragraph we
survey only a few essential ideas, with a focus
on explicit algorithms.

Finding a solution to polynomial system (\ref{eq:systemp})
amounts to finding a point in the set
\[
{\mathscr V} := \{f \in {\mathbb C}^n \: :\: p_1(f) = 0, \ldots, p_n(f) = 0\}.
\]
Set $\mathscr V$ is a subset of ${\mathbb C}^n$ called a (complex)
algebraic variety because it is the vanishing locus of a finite number of
polynomials. To the geometric object $\mathscr V$ corresponds
an algebraic object:
\[
{\mathscr I} := \{a_1(f)p_1(f) + \cdots + a_n(f)p_n(f) \: :\:
a_1, \ldots, a_n \in {\mathbb Q}[f]\}
\]
which is a subset of ${\mathbb Q}[f]$ called an algebraic ideal.
Elements in $\mathscr I$ are obtained by taking linear combinations
(with polynomial coefficients) of polynomials $p_i$.
We say that ideal $\mathscr I$ is generated by $p_1,\ldots,p_n$, and we
say that $p_1,\ldots,p_n$ is a presentation of $\mathscr I$.
Although $\mathscr I$ has an infinite number of elements, it follows from a
fundamental theorem of Hilbert \cite[Theorem 4, \S 2.5]{cox}
that it has always a finite presentation,
and this allows $\mathscr I$ to be handled by a computer.
Note that every polynomial in $\mathscr I$ vanishes at points $f \in {\mathscr V}$,
and we say that $\mathscr V$ is the variety associated to the ideal $\mathscr I$.

A key idea of algebraic geometry consists in obtaining
useful information on $\mathscr V$ 
from a suitable presentation of $\mathscr I$.
By taking finitely many linear combinations of polynomials $p_i$,
we will generate another equivalent system of polynomials
which generates the same ideal $\mathscr I$ but with another
presentation, and which is associated with the same variety $\mathscr V$.
In particular, when $\mathscr V$ is a discrete set, i.e. when
Assumption~\ref{as:finitesol} is satisfied, this presentation should allow
to compute the solutions easily. A useful presentation
is a Gr\"obner basis \cite[Section 2]{cox}. To obtain such a basis, we can devise
an algorithm using only linear algebra and performing a series
of multivariate polynomial divisions. These divisions
can be carried out provided one defines a suitable
ordering on the set of monomials of variables $f_1,\ldots,f_n$.
For our purpose of computing the solutions, a useful ordering
is the graded reverse lexicographic (grevlex) order, see
\cite[Definition 6 of Section 2.2]{cox}.
Once a Gr\"obner basis is available, we can compute a rational
univariate representation (RUR) 
{\color{black}
\begin{equation}\label{rur}
%\begin{array}{rcl}
r(f_{\circ}) = 0,\;
f_1 =  \frac{r_1(f_{\circ})}{r_0(f_{\circ})},\; \ldots\; ,
f_n  =  \frac{r_n(f_{\circ})}{r_0(f_{\circ})},
%\end{array}
\end{equation}
}
where $r$ is a suitable 
univariate polynomial in a variable $f_{\circ}$, and $r_0,\ldots,r_n$ are univariate
polynomials of degree less than the degree of $r$.
Variable $f_{\circ}$ is called the separating variable, and it is
a linear combination of the original variables $f_1,\ldots,f_n$.

\begin{proposition}
For system (\ref{eq:systemp}), a  Gr\"obner basis always
exists. Moreover, all but  
a finite number of linear combinations of the variables $f_i$,
$i=1,\ldots, n$ are separating variables. Finally,
once the separating variable is chosen, the RUR exists and is unique.
\end{proposition}

\begin{proof}
The existence of a Gr\"obner basis follows from \cite[Corollary 6, \S
2.5]{cox}. The rest of the proposition can be proven using
 \cite{fglm93}, \cite[Theorem 3.1]{rou99} or \cite[Proposition 12.16]{bpr06}.
See also \cite[Proposition 2.3]{stu02} and the discussion just after,
which explains the connection between RUR and Gr\"obner basis.
\end{proof}

Once a RUR is available, enumerating all the real solutions amounts to computing
all the real roots of the univariate polynomial $r$ in (\ref{rur}), and evaluating
the rational functions $r_i/r_0$ in (\ref{rur}) at those points. By an appropriate
use of Descartes' rule of signs and Sturm sequences, see e.g.
\cite[Section 2.2]{bpr06}, an algorithm can be designed that isolates all real roots
of the univariate polynomial in rational intervals
of arbitrarily small width.

\subsection{Numerical results and sensitivity analysis}
\label{sec:alg_sens}

The simplest possible selection of parameters $\alpha_i$ in the
assignment problem (\ref{eq:assignBF}) is to select them as the smallest possible set
of distinct even and positive integers, namely $\alpha_i = 2i$ for all $i\in
\{1, \ldots, n\}$. This selection gives rise to the smallest possible
coefficient list in (\ref{eq:systemp}) and
leads to a set of solutions to Problem~\ref{prob_stat}
having the smallest possible maximal frequency of the natural resonant
modes of the circuit. This was also the preferred solution addressed
in \cite{ZackPPC09,BuchenauerTPS10}. In particular, for this specific
selection of the resonant frequencies, a set of 10 solutions for the case
$n=6$ has been given in \cite[Table II]{BuchenauerTPS10}. The
advantage of the formal approach of this section is to allow to find the
complete set of solutions, amounting to 12 (2 solutions were missed in
\cite[Table II]{BuchenauerTPS10} which used bounded random sampling
followed by numerical optimization). Another advantage of our results is that
the sets of solutions reported here were computed in a few minutes
using the Maple code reported in Appendix~\ref{sec:maple}.
Table~\ref{tab:allnumbers} reports the numerical values computed by
the analytic algorithm proposed in this section. Note that the
displayed values correspond to $n^2 c_i/c = n^2 f_i^{-1}$. We choose
to represent these values to allow for an easier comparison with the
results in \cite[Table II]{BuchenauerTPS10} and because they are
better numerically conditioned.
Note that solutions number 1 and 2 for the case $n=6$ were not
reported in \cite[Table II]{BuchenauerTPS10}.

%$1$ & $2.662883$ & $8.449490$ & $1.110$ \\
%$2$ & $6.337117$ & $3.550510$ & $1.027$   

\newcommand\rr[1]{\rule{0cm}{0.3cm} {\normalsize{#1}}}
\renewcommand{\tabcolsep}{.1cm}

\begin{table}[ht!]
{
\begin{center}
\begin{tabular}{|c|cccccc|c|}
\hline \rule{0cm}{.6cm}
 & {\normalsize $n^2\dfrac{c_1}{c}$} & {\normalsize $n^2\dfrac{c_2}{c}$} & {\normalsize $n^2\dfrac{c_3}{c}$}
 & {\normalsize $n^2\dfrac{c_4}{c}$} & {\normalsize $n^2\dfrac{c_5}{c}$} & {\normalsize $n^2\dfrac{c_6}{c}$}
 & cond \\[.3cm] \hline 
% & {\normalsize $f_1$} &{\normalsize $f_2$} & {\normalsize $f_3$} &
% {\normalsize $f_4$} & {\normalsize $f_5$} & {\normalsize $f_6$} & sens \\[.1cm] \hline 
%  $\dfrac{c}{c_2}$&  $\dfrac{c}{c_3}$&
% $\dfrac{c}{c_4}$&  $\dfrac{c}{c_5}$ & \\[.3cm] \hline
\multicolumn{2}{|l}{\rr{$n=1$}} & & & & & & \\
$1$ & $1.5$ & & & & & & $1$ 
 \\ \hline
\multicolumn{2}{|l}{\rr{$n=2$}} & & & & & & \\
$1$ & $1.50213$ & $0.47340$ & & & & & $1.1102$ \\
$2$ & $0.63120$ & $1.12660$ & & & & & $1.0266$   
 \\ \hline
\multicolumn{2}{|l}{\rr{$n=3$}} & & & & & & \\
$1$ & $1.49303$ & $1.49229$ & $0.41548$ & & & & $1.1557$ \\
$2$ & $0.84408$ & $0.77662$ & $1.41217$ & & & & $1.0387$   
 \\ \hline
\multicolumn{2}{|l}{\rr{$n=4$}} & & & & & & \\
$1$ & $1.71070$ & $1.29555$ & $1.54667$ & $0.38529$ & & & $1.1849$ \\
$2$ & $1.62637$ & $0.62519$ & $1.73498$ & $0.74862$ & & & $1.0917$ \\
$3$ & $1.06181$ & $2.10211$ & $0.66491$ & $0.89099$ & & & $1.121$ \\
$4$ & $1.13210$ & $0.78731$ & $0.92450$ & $1.60306$ & & & $1.0440$   
 \\ \hline
\multicolumn{2}{|l}{\rr{$n=5$}} & & & & & & \\
$1$ & $2.04567$ & $1.23900$ & $1.35694$ & $1.57111$ & $0.36796$ & & $1.2018$ \\
$2$ & $2.14782$ & $0.63778$ & $1.24610$ & $1.70028$ & $0.68506$ & & $1.1092$ \\
$3$ & $0.99720$ & $1.69936$ & $1.57266$ & $0.64267$ & $1.16088$ & & $1.0558$ \\
$4$ & $1.47480$ & $0.86342$ & $0.84481$ & $1.07344$ & $1.72179$ & & $1.0448$   
\\ \hline
\multicolumn{2}{|l}{\rr{$n=6$}} & & & & & & \\
$1$ & $2.49095$ & $1.25588$ & $1.20240$ & $1.49359$ & $1.53290$ & $0.35987$ & $1.2065$ \\
$2$ & $1.92537$ & $1.79971$ & $1.80083$ & $0.90696$ & $1.45858$ & $0.37545$ & $1.1954$ \\
$3$ & $1.67555$ & $1.98118$ & $2.05786$ & $0.66667$ & $1.30433$ & $0.52174$ & $1.112$ \\
$4$ & $2.65073$ & $1.01602$ & $0.68610$ & $1.82261$ & $1.42150$ & $0.64738$ & $1.1506$ \\
$5$ & $1.34706$ & $2.18478$ & $0.92044$ & $1.84208$ & $0.65432$ & $0.94921$ & $1.0971$ \\
$6$ & $1.95229$ & $0.93587$ & $1.53272$ & $0.63062$ & $1.76898$ & $0.99206$ & $1.0844$ \\
$7$ & $2.46541$ & $0.73028$ & $0.99423$ & $0.93809$ & $1.85839$ & $0.99314$ & $1.1223$ \\
$8$ & $1.79820$ & $0.94167$ & $1.74742$ & $0.62528$ & $1.59040$ & $1.05327$ & $1.0884$ \\
$9$ & $1.43355$ & $1.89698$ & $0.60976$ & $1.73302$ & $1.02388$ & $1.05334$ & $1.1009$ \\
$10$ & $1.50458$ & $1.00778$ & $2.04896$ & $1.05339$ & $0.68756$ & $1.37732$ & $1.0686$ \\
$11$ & $1.38734$ & $1.13092$ & $1.48392$ & $1.32482$ & $0.63764$ & $1.57578$ & $1.0617$ \\
$12$ & $1.87892$ & $0.96056$ & $0.85587$ & $0.91518$ & $1.23619$ &
$1.77345$ & $1.0592$    \\ \hline 
\end{tabular}
\end{center}
}
\caption{The solutions to the Marx design problem computed using
  Gr\"obner basis methods and their conditioning.}
\label{tab:allnumbers}
\end{table}

According to the observations reported in Remark~\ref{rem:sensitivity}
we can characterize the sensitivity of each solution obtained from the
proposed symbolic solution method by looking at the condition number of each
eigenvalue of matrix $BF$ in (\ref{eq:assignBF}). This method can be
applied because all eigenvalues of $BF$ are distinct (therefore
simple) by assumption. In particular, for each eigenvalue of $BF$ its condition
 number corresponds to 
$|w^T v|^{-1}$, where $w$ and $v$ have unit norm and 
are respectively the left and right
eigenvectors associated with that eigenvalue.
The results of the sensitivity analysis is represented by the maximum
condition number among all eigenvalues of matrix $BF$ and is shown in the
last column of Table~\ref{tab:allnumbers} for each one of the computed
solutions. 

Inspecting the different sensitivities it appears that the last
solution for each one of the analyzed cases corresponds to the least
sensitive one, namely the one that is expected to be more robust.
Interestingly, this solution corresponds to the solution qualitatively
characterized in \cite{BuchenauerTPS10} as the ``regular''
solution. Indeed, when looking at the time responses of the Marx generator
network designed with these parameters, one experiences little
dependence on higher order harmonics and some suitable monotone evolution
 of certain voltages in the circuit
(see \cite[Fig. 10]{BuchenauerTPS10} for an example of this). 
Another peculiar feature of the ``regular'' solutions corresponding to
the last solution for each $n$ in Table~\ref{tab:allnumbers} is that
the interpolated values of $n^2 c_i/c$ form a convex function of $i$,
namely (since $n^2/c$ is constant) one has 
$c_j \leq  \frac{c_{j+1}+c_{j-1}}{2}$ for all $j=2,\ldots,n-1$
(see also the red curve in \cite[Fig. 8]{BuchenauerTPS10}
corresponding to the last solution for $n=6$ in
Table~\ref{tab:allnumbers}). 
Moreover, at least up to $n=6$, numerical evidence reveals that there
only exists one such solution. Due to its desirable features both in
terms of numerical robustness and of desirable time evolution of the
arising circuit (as reported in \cite{BuchenauerTPS10}), we will be
imposing this constraint on the numerical optimization described in
the next section, to be able to isolate that specific solution for the
case $n>6$ (or all of such specific solutions, if more than one of
them exist).

\begin{figure}[ht!]
\begin{center}
\includegraphics[width=\columnwidth]{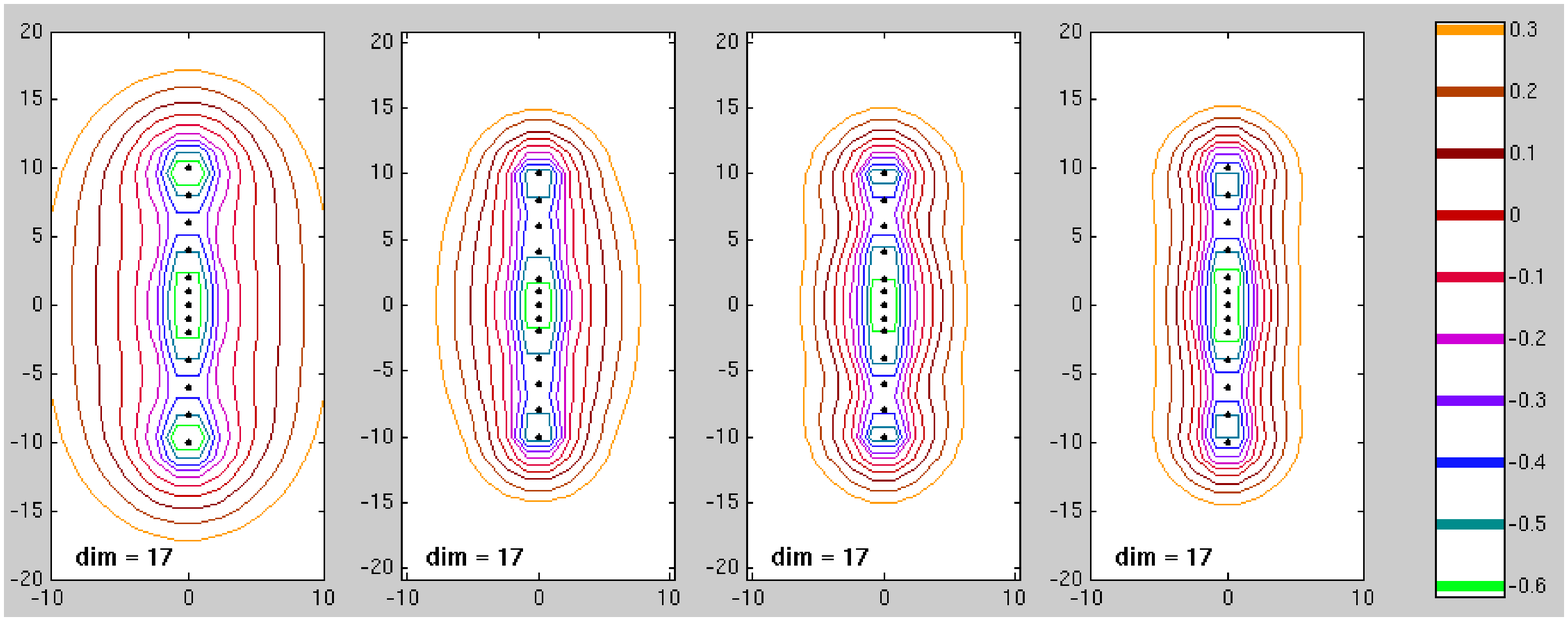}
\end{center}
\caption{The pseudospectra of the state transition matrix $A_0$ for the
  four solutions corresponding to $n=5$.}
\label{fig:pseudo}
\end{figure}

A final comment should be made about the use of the 
pseudospectra of matrix $BF$ as a graphical tool to assess the
sensitivity of each solution of Table~\ref{tab:allnumbers} to larger
perturbations. 
According to the results in  \cite{TreEmb05} and the
corresponding Matlab tool Pseudospectra GUI available in the EigTool
package \cite{TreEmb05}, given $\epsilon>0$,
the associated
pseudospectrum of a matrix $A_0$ corresponds to the following region
in the complex plane:
$$
\{s\in \complex \!\!: \; \exists \Delta\in\complex^{n\times n},
\det(sI_n - A -\Delta) \!=\!0, \mbox{ and } 
\|\Delta\| \leq \epsilon \}
$$
(see Figure~\ref{fig:pseudo} for some examples).
Surprisingly, the graphical aspect of the pseudospectra
appear indistinguishable among all solutions for fixed
$n$. Nevertheless, interesting results are obtained when applying this
analysis to the state transition matrix $A_0$ associated to the
state-space model of the circuit (for fixed values of $\ell$ and
$c$) whose expression is given  in
(\ref{eq:A0}), in Section~\ref{sec:model}. 
An example
of the type of pseudospectra obtained for the case $n=5$ is
represented in Figure~\ref{fig:pseudo} corresponding to the case $c=\ell=1$. 
The figure shows that once again the ``regular''
solution is associated with the least sensitive scenario. 
Visually, this corresponds to the tightest pseudospectra (see the rightmost case in 
Figure~\ref{fig:pseudo}) which can be best appreciated by inspecting
the largest level set of the figure, corresponding to the selection
$\epsilon = 10^{0.3}$ according to the legend to the right of the figure.
The other values of $n$ lead
to similar results.

%{\bf DIDIER:
%here in this subsection we should report on results
%for the symbolic solutions for say $n=6$, observing
%that the most regular (to be defined) solution
%is the one that corresponds to a convex shape (to
%be defined). This motivates the introduction (in
%the next section) of an objective function
%and convexity constraints (to be done).}

\section{Numerical solution}
\label{sec:numerical}

In this section we use convex optimization techniques
to find numerically the real solutions of our polynomial
system of equation. 
As compared to the previous section, 
we focus here on the alternative formulation \eqref{eq:systemq}
because for this formulation, according to Lemma~\ref{lem:Kbox},
all the
 real solutions satisfy
$|k_i| \leq 1$, $i=1,2,\ldots,n$. As explained
e.g. in \cite{HenLas05}, for numerical reasons
it is very important that the problem unknowns are
scaled down to the unit interval.

\subsection{Problem formulation}
\label{sec:num_probstat}

A numerical approach solution to the inverse eigenvalue problem
presented in Section~\ref{sec:twoform}
 consists
in formulating it first as a nonconvex polynomial optimization problem:
\begin{equation}\label{polyopt}
\begin{array}{rclll}
q^* & = & \min_k & q_0(k) \\
&& \mathrm{s.t.} & k \in {\mathscr K}
\end{array}
\end{equation}
where the objective function $q_0 \in {\mathbb Q}[k]$ is a given polynomial of the
vector of indeterminates $k \in {\mathbb R}^n$, and,
based on (\ref{eq:systemq}) and on the discussion about ``regular
solutions'' in Section~\ref{sec:alg_sens},
 the feasibility
set
\[
\begin{array}{l}
{\mathscr K} = \{k \in {\mathbb R}^n \: :\: q_i(k) = 0, \:\: i=1,\ldots,n, \\
\quad\quad\quad g_j(k) := k_{j}-2k_{j+1}+k_{j+2} \geq 0, \:\: j=1,\ldots,n-2\}
\end{array}
\]
is the real algebraic variety corresponding to the zero locus
of the ideal $\mathscr I$ studied in Section~\ref{sec:symbolic}
intersected with the polyhedron modeling the convexity
constraints introduced in Section~\ref{sec:alg_sens}.

A typical objective function in problem (\ref{polyopt})
can be the positive definite convex quadratic form
\[
q_0(k) = \sum_{i,j=1}^n (k_i-k_j)^2 = k^T
\smallmat{
		2 & -1 & 0 & \cdots & 0\\
		-1 & \ddots & \ddots & \ddots & \vdots\\
		0 & \ddots & \ddots & \ddots & 0\\
		\vdots & \ddots & \ddots & 2 & -1 \\
		0 & \cdots & 0 & -1 & 2
}
k
\]
so that capacitors $c_i = k_i/c$ are as identical as possible,
but we can also consider other relevant objective functions,
not necessarily quadratic, definite in sign or convex.

Optimization problem (\ref{polyopt}) is finite-dimensional, algebraic,
but nonconvex since the feasibility set $\mathscr K$ is disconnected,
as a union of isolated points (because of Assumption~\ref{as:finitesol}). Local optimization techniques
based on nonlinear programming are likely to face troubles with
such sets.
Function $q_0$ is continuous and we optimize it over
$\mathscr K$ which is compact, since by Assumptions~\ref{as:finitesol}
and~\ref{as:onesol}
 there is at least one real solution
and at most a finite number of isolated real solutions.
It follows that optimization problem (\ref{polyopt})
has at least one solution.
Since $q_0$ is not necessarily convex and $\mathscr K$ is
disconnected, we do not expect optimization problem (\ref{polyopt})
to have a unique global minimizer.
However, we expect the number of global minimizers to be
significant smaller than the cardinality of set $\mathscr K$.

\subsection{Optimization method using Gloptipoly}

Our optimization method is based on an idea first described in \cite{Las00}
which consists in reformulating a nonconvex global optimization problem with polynomial data
(i.e. minimization of a polynomial objective function subject to
polynomial inequalities and/or equations) as an equivalent convex linear programming (LP) problem
over probability measures. Instead of optimizing over a vector in a finite-dimensional
Euclidean space, we optimize over a probability measure in an infinite-dimensional
Banach space. The measure is supported on the feasibility set of the optimization problem,
which is algebraic in our case, and we require in addition that the set is bounded
(which is true in our case, by assumption).
More concretely, a probability measure is understood as a linear functional
acting on the space of continuous functions, and we manipulate a measure
through its moments, which are images of monomials (which are dense w.r.t. the supremum norm
in the space of continuous functions with compact support). Using results on functional analysis
and real algebraic geometry, and under some mild assumption on the compact support,
a sequence of real numbers are moments of a probability measure if they belong
to an appropriate affine section of the cone of positive semidefinite linear operators,
an infinite-dimensional convex set. We then construct a hierarchy of finite-dimensional
truncations of this convex set, namely affine sections of the cone of positive semidefinite
matrices of fixed size. As a result, solving an LP in the set of probability measures
with compact semi-algebraic support boils down to solving a hierarchy of semidefinite
programming (SDP) problems, also called linear matrix inequalities (LMIs).

When there is a finite number of global optimizers, the approach is guaranteed to converge in a finite number
of steps, and the global optimizer(s) can be extracted with the help of numerical linear
algebra, see \cite{HenLas05}. We have then a numerical certificate of global optimality of the solution(s).
This approach has been successfully applied to solve globally various polynomial optimization
problems, see \cite{Lau09} and \cite{Las09} for general overviews of results and applications.
For applications in systems control, the reader is referred to the survey \cite{HenLas04}.
For finding real solutions of systems of polynomial equations and real radical ideals,
the approach has been comprehensively studied in \cite{LasLauRos08}.

More explicitly, we now describe our approach to the numerical solution of problem (\ref{polyopt}).
We consider a compact set ${\mathscr K} \subset {\mathbb R}^n$ and we denote by
${\mathscr M}({\mathscr K})$ the Banach space of
Borel measures supported on $\mathscr K$. These are nonnegative functions from the Borel sigma-algebra
of subsets of $\mathscr K$ to the real line ${\mathbb R}$.  Given a measure
$\mu \in {\mathscr M}({\mathscr K})$ we define
its moment of order $\alpha \in {\mathbb N}^n$ as the real
number~\footnote{The notation $\alpha_i$ is used in (\ref{moment}) and
  the remaining derivations in this section, for
  consistency with the notation used in \cite{Las00} and references
  therein. However, they should not be confused with the scalars
  $\alpha_i$ used in Theorem~\ref{th:design}.}
\begin{equation}\label{moment}
y_{\alpha} = \int_{\mathscr K} x^{\alpha} \mu(dx) \in \mathbb R
\end{equation}
where we use the multi-index notation for monomials, i.e. $x^{\alpha} = \prod_{i=1}^n x^{\alpha_i}_i$.
We define the infinite-dimensional vector $y=\{y_{\alpha}\}_{\alpha \in {\mathbb N}^n}$ as the sequence
of moments of $\mu$.
Note that $y_0 = \int \mu = \mu({\mathscr K}) = 1$ whenever $\mu \in {\mathscr M}({\mathscr K})$
is a probability measure.
Moreover, if for each $k \in {\mathscr K}$, $|k_i|\leq 1$ for all $i$
(this is what we establish in Lemma~\ref{lem:Kbox}), then
$|y_{\alpha}|\leq 1$ for all $\alpha \in {\mathbb N}^n$. Conversely,
for larger sets ${\mathscr K}$, the variable $y_{\alpha}$ may grow
very large and this is not convenient for numerical reasons. This
aspect has been pointed out in  \cite{HenLas05} and motivates Lemma~\ref{lem:Kbox}.

Given a sequence $y$, we define the Riesz linear functional $\ell_y : {\mathbb R}[x] \to {\mathbb R}$
which acts on polynomials $\pi(x) = \sum_{\alpha} \pi_{\alpha} x^{\alpha}$
as follows: $\ell_y(\pi(x)) = \sum_{\alpha}
\pi_{\alpha} y_{\alpha}$. If sequence $y$ has a representing measure $\mu$, integration of polynomial $\pi(x)$
w.r.t. $\mu$ is obtained by applying the Riesz functional $\ell_y$ on $\pi(x)$, since $\ell_y(\pi(x)) = \int \pi(x)\mu(dx) =
\int \sum_{\alpha} \pi_{\alpha} x^{\alpha} \mu(dx) = \sum_{\alpha} \pi_{\alpha} \int x^{\alpha} \mu(dx) =
\sum_{\alpha} \pi_{\alpha} y_{\alpha}$.

If we apply the Riesz functional on the square of a polynomial $\pi(x)$ of degree $d$, then we obtain a form which
is quadratic in the coefficient vector $\pi=\{\pi_\alpha\}_{|\alpha|\leq d}$ and which we denote
\[
\ell_y(\pi^2(x)) = \pi^T M_d(y) \pi
\]
where $M_d(y)$ is a symmetric matrix which is linear in $y$, called the moment matrix of order $d$.
Rows and columns in this matrix are indexed by vectors $\alpha \in {\mathbb N}^n$ and $\beta \in {\mathbb N}^n$,
and inspection reveals that indeed the entry $(\alpha,\beta)$ in matrix $M_d(y)$ is the moment $y_{\alpha+\beta}$.
Given a polynomial $\chi(x)$ we let
\[
\ell_y(\pi^2(x)\chi(x)) = \pi^T M_d(\chi,y) \pi
\]
where $M_d(\chi,y)$ is a symmetric matrix which is linear in $y$ and linear in coefficients of $\chi(x)$,
called the localizing matrix of order $d$ w.r.t. $\chi(x)$. The localizing matrix is a linear combination
of moment matrices, in the sense that entry $(\alpha,\beta)$ in $M_d(\chi,y)$ is equal to
$\sum_{\gamma} \chi_{\gamma} y_{\alpha+\beta+\gamma}$.

\subsection{Application to the eigenvalue assignment problem}

Based on the optimization method presented in the previous section, we
now formulate the polynomial optimization problem (\ref{polyopt}) as
a hierarchy of finite-dimensional LMI problems with the help of the
Matlab interface GloptiPoly 3 \cite{HenLas09}. Then, we use public-domain
implementations of primal-dual interior point algorithms to solve the
LMI problems. These algorithms rely on a suitable logarithmic barrier function for
the SDP cone, and they proceed by iteratively reducing the duality gap between
the primal problem and its dual, which is also an LMI problem. Each iteration
consists in solving a Newton linear system of equations, involving the gradient
and the Hessian of a Lagrangian built from the barrier function. Most of the
computational burden comes from the construction and the storage of the
Hessian matrix, and problem sparsity can be largely exploited at this stage.
For more information on SDP and related optimization methods, see e.g. \cite{BenNem01}.
For our numerical examples we have been using the SDP solver
SeDuMi 1.3 \cite{Stu99}.

More specifically,
let $d_i=\lceil\frac{\deg q_i}{2}\rceil$, $i=0,1,\ldots,n$, and consider the optimization problem
\begin{equation}\label{lmirelax}
\begin{array}{rcll}
q^*_d & = & \inf_y & \ell_y(q_0) \\
&& \mathrm{s.t.} & y_0 = 1 \\
&&& M_d(y) \succeq 0 \\
&&& M_{d-d_i}(q_i,y) = 0, \quad i=1,\ldots,n \\
&&& M_{d-1}(g_j,y) \succeq 0, \quad j=1,\ldots,n-2
\end{array}
\end{equation}
for $d \geq \max\{d_i\}_{i=0,1,\ldots,n}$,
where $\succeq 0$ stands for positive semidefinite.

In the above problem, the unknown is
the truncated sequence $y$ of moments of degree up to $2d$, and the constraints are convex
linear matrix inequalities (LMI) in $y$. Problem (\ref{lmirelax}) is called the LMI relaxation
of order $d$ of problem (\ref{polyopt}). It can be proved that the
infimum in LMI problem (\ref{lmirelax}) is attained, as stated in the
next proposition which is proven in \cite[Theorem 6.1]{Las09}.

\begin{proposition}\label{lmiglobal}
The optimal values of (\ref{lmirelax}) satisfy
$q^*_{d} \leq q^*_{d+1}$. Moreover, under Assumptions~\ref{as:finitesol}
and~\ref{as:onesol} there exists a finite $d^* \in {\mathbb N}$
such that $q^*_d=q^*$ for all $d\geq d^*$,
where $q^*$ is the optimal value of (\ref{polyopt}).
\end{proposition}

%{\bf Proof:} See \cite[Theorem 6.1]{Las09}. $\Box$

{\color{black}
Roughly speaking, Proposition~\ref{lmiglobal} establishes that the
solutions to the 
sequence of relaxations
 converge at a
finite (although unknown) value of $d=d^*$, so that solving
LMI (\ref{lmirelax}) is equivalent to solving problem (\ref{polyopt}). }
We should remark that while LMI solutions
are not very accurate, they can be obtained cheaply (at least for
these examples) since we do not need to enumerate all real
solutions, only the optimal one(s). Moreover, if a computed
solution is not deemed accurate enough, it can be refined
locally afterwards by Newton's method if required.

\begin{remark}
\label{rem:relax}
Proposition~\ref{lmiglobal} states that solving nonconvex polynomial optimization (\ref{polyopt})
is equivalent to solving convex LMI problem (\ref{lmirelax}) for a sufficiently large
relaxation order $d^*$
if the feasibility set has zero dimension and is nonempty
. Unfortunately, it is not possible to give a priori useful (lower or upper)
bounds on $d^*$,
and the strategy followed in \cite{HenLas05} is to detect global optimality of
an LMI relaxation by inspecting the rank of the moment matrix and then extract
the globally optimal solutions by linear algebra, see also \cite{LasLauRos08}
and \cite[Algorithm 6.1]{Las09}.
For all the Marx generator design problems that we considered, we
observed that the global optimum is certified at the smallest possible
LMI relaxation, i.e., in Proposition~\ref{lmiglobal}, 
$
d^* = \max\nolimits\limits_{i=0,1,\ldots,n}\left\lceil\frac{\deg q_i}{2}\right\rceil
=\left\lceil\frac{n}{2}\right\rceil . 
$
\end{remark}

\subsection{Numerical results and sensitivity analysis}

We applied this numerical approach on our examples, with
different objective functions, and the overall
conclusion is that the cases $n=2,3,4,5$ are solved very
easily (in a few seconds) but the solution
(obtained with SeDuMi) is not very accurate. The case $n=6$ is
solved in a few minutes, and the case $n=7$
is significantly harder: it takes a few hours to be solved
and it provides the following ``regular'' solution (in the sense
introduced in Section~\ref{sec:alg_sens}):

{
\begin{center}
\begin{tabular}{|c|c|c|c|c|c|c|}
\hline \rule{0cm}{.6cm}
 {\normalsize $n^2\dfrac{c_1}{c}$} & {\normalsize $n^2\dfrac{c_2}{c}$} & {\normalsize $n^2\dfrac{c_3}{c}$}
 & {\normalsize $n^2\dfrac{c_4}{c}$} & {\normalsize $n^2\dfrac{c_5}{c}$} & {\normalsize $n^2\dfrac{c_6}{c}$}& {\normalsize $n^2\dfrac{c_7}{c}$} \\[.3cm] \hline 
\rule{0cm}{.3cm}
$2.07061$ & $1.05669$ & $1.04940$ & $1.05715$ & $1.06861$ & $1.08449$ & $1.85298$ \\  \hline
\end{tabular}
\end{center}
}

For this solution, we can compute the sensitivity level using the same
algorithm used in the last column of Table~\ref{tab:allnumbers} and we
obtain $1.0502$.
Finally, solving the case $n=8$ takes approximately $15$ hours and leads
to the following set of parameters. The sensitivity of this solution 
corresponds to $1.0617$.
{\color{black}
Figure~\ref{fig:sim} shows a time history of the corresponding
response
with $c=\ell=1$, with the notation introduced later in Figure~\ref{fig:circuit2}.
The simulation shows that all the energy initially stored in the
storage capacitors is transferred to the load capacitor (note that the
black curve in the lower plot represents $v_L/n$).
Note also that the eight storage and parasitic voltages are
characterized by an ordering which remains constant along the whole
trajectory, which is a peculiar feature of the so-called ``regular''
solution (note that a similar behavior is obtained for the regular
solution associated with $n=6$, as reported in \cite[Fig. 10]{BuchenauerTPS10}).
}

{
\begin{center}
\begin{tabular}{|c|c|c|c|c|c|c|c|}
\hline \rule{0cm}{.6cm}
 {\normalsize $n^2\dfrac{c_1}{c}$} & {\normalsize $n^2\dfrac{c_2}{c}$} & {\normalsize $n^2\dfrac{c_3}{c}$}
 & {\normalsize $n^2\dfrac{c_4}{c}$} & {\normalsize $n^2\dfrac{c_5}{c}$} & {\normalsize $n^2\dfrac{c_6}{c}$}& {\normalsize $n^2\dfrac{c_7}{c}$} & {\normalsize $n^2\dfrac{c_8}{c}$}\\[.3cm] \hline 
\rule{0cm}{.3cm}
$2.39407$ & $1.17326$ & $1.12475$ & $1.11221$ & $1.10440$ & $1.09960$
& $1.09985$ & $1.87282$  \\  \hline
\end{tabular}
\end{center}
}

\begin{figure}[ht!]
\begin{center}
\includegraphics[width=\columnwidth]{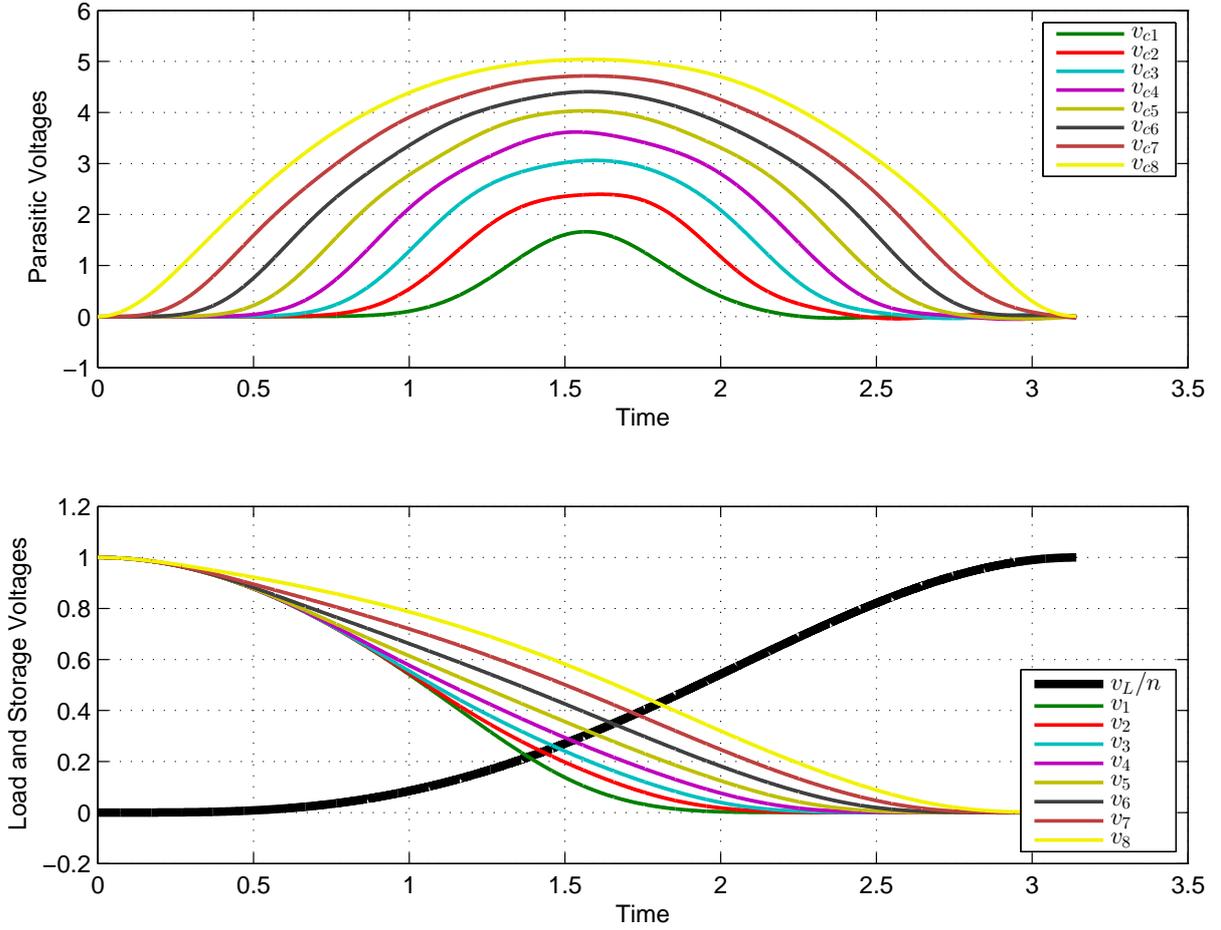}
\end{center}
\caption{{\color{black} Time histories of the parasitic capacitor voltages (upper plot) and
  of the load and storage capacitors voltages (bold and solid curves
  in lower plot) for the regular solution to the case $n=8$ 
  found via numerical optimization.}}
\label{fig:sim}
\end{figure}

\section{Circuit description and proofs}
\label{sec:proofs}

We carry out the proof of Theorem~\ref{th:design} by first providing a
mathematical description of the circuit (Section~\ref{sec:model}),
then proving a useful resonance result (Section~\ref{sec:reson}) and
then proving the theorem (Section~\ref{sec:proof}).

\subsection{Circuit description}
\label{sec:model}

%It is well known that passive linear electrical networks can be described by state-space
%models where the state comprises voltages across capacitors and
%currents flowing in the inductances. Therefore, i
Following an approach similar to the one adopted in \cite{Antoun06},
it is possible to
derive a state-space representation of the Marx generator of
Figure~\ref{fig:circuit} using $3n+2$ state variables comprising $2n+1$ voltages
across the $2n+1$ circuit capacitors, and $n+1$ currents flowing in the $n+1$ inductors.

\begin{figure}[ht!]
\begin{center}
\includegraphics[width=\columnwidth]{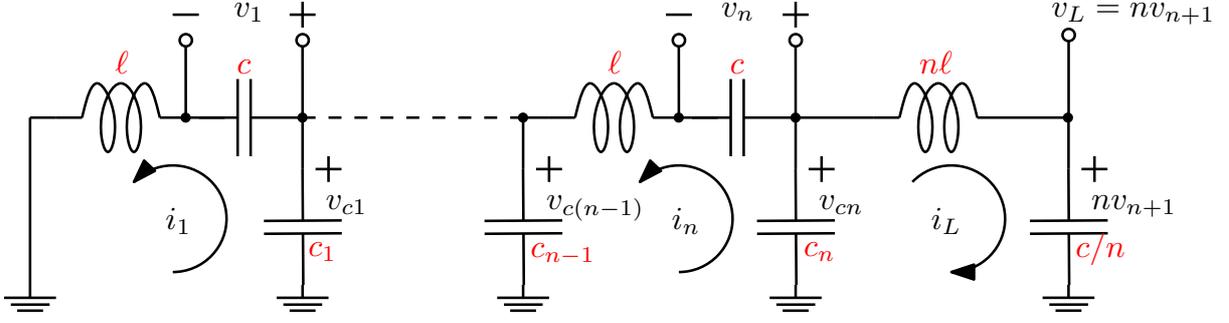}
\end{center}
\caption{Sign conventions in the models of the Marx generator circuit.}
\label{fig:circuit2}
\end{figure}

In particular, using the sign convention and the component values
depicted in Figure~\ref{fig:circuit2}
(see also Figure~\ref{fig:circuit}), 
we can define the state variable as 
{\color{black}
\begin{equation}
\label{eq:state}
x = [v_{c1}, \cdots, \; v_{cn}, \; v_1, \cdots, v_n,\; v_{n+1},\; i_1, \cdots, i_n, \; i_L]^T,
\end{equation}
}
{\color{black} (note that $i_L$ has opposite direction to the other
  currents to simplify the derivations in the proof of Theorem~\ref{th:design})}
and the linear dynamics of the circuit correspond to the equations 
%$c \dot v_k = i_k$, $k=1,\ldots, n$, 
%$c \dot v_{n+1} = i_L$, 
%$L\dot i_1 = v_{c1} - v_1$, 
%$L\dot i_k = v_{ck} - v_{c(k-1)} - v_k$, $k=2,\ldots, n$,
%$n L\dot i_L = v_{cn} - n v_{n+1}$, 
%$c_k \dot v_{ck} = i_{k+1} - i_k$, $k=1,\ldots, n-1$, 
%$c_n \dot v_{cn} = -i_L - i_n$, 
\begin{align*}
&c \dot v_k = i_k,\;  k=1,\ldots, n, 
\quad c \dot v_{n+1} = i_L, \\ 
& \ell\dot i_1 = v_{c1} - v_1, \quad n \ell \dot i_L = v_{cn} - n v_{n+1}, \\
& \ell\dot i_k = v_{ck} - v_{c(k-1)} - v_k,\;  k=2,\ldots, n,\\
& c_k \dot v_{ck} = i_{k+1} - i_k,\; k=1,\ldots, n-1, \quad
c_n \dot v_{cn} = -i_L - i_n, 
\end{align*}
which can be written in compact form using the following
linear state-space model $\dot x =A_0 x$ where 
{\color{black}
\begin{equation}
\label{eq:A0}
A_0:= \mymatrix{c|c}{0_{(2n+1)\times(2n+1)} & 
  \begin{array}{c}-\frac{1}{c} F\Sigma  J_{-1}\\[.2cm] \frac{1}{c} I_{n+1}  \end{array} \\ \hline \\[-.3cm]
  \begin{array}{cc}\frac{1}{\ell} J_{n+1}^{-1} J_{-1} \Sigma^T &-\frac{1}{\ell} I_{n+1}\end{array} &
  0_{(n+1)\times(n+1)}
    },
\end{equation}
where $0_{q \times q}\in \real^{q\times q}$ is a square matrix of zeros, 
$I_q$ is the identity matrix of size $q$,
$J_{n+1} = \diag(1, \ldots, 1, n) \in \real^{(n+1)\times (n+1)}$,
$J_{-1} = \diag(1, \ldots, 1, -1) \in \real^{n\times n}$
 are both diagonal matrices,
$\Sigma = \smallmat{1 & -1 & 0 & \cdots & 0\\ 
                    0 & 1  & -1 & \cdots & 0\\
                    \vdots & \vdots & \ddots & \ddots & \vdots \\
                    0 & \cdots & 0 & 1 & -1} \in \real^{n\times
                    (n+1)}$ and $F$ is defined in the statement of Theorem~\ref{th:design}.
}

\subsection{A sufficient resonance condition}
\label{sec:reson}

In this section we establish a preliminary result that will be useful for the
proof of Theorem~\ref{th:design} and which is based on the
description (\ref{eq:A0}) of the generator dynamics.  

{\color{black}

\begin{lemma}
\label{lem:Wpoles}
Consider the Marx circuit in Figure~\ref{fig:circuit} and the matrices
$B$ in (\ref{eq:matrix_B}) and $F = \diag(c/c_1, \ldots , c/c_n)$. 
The matrix $I+  B F$ has $n$ real positive eigenvalues. 
Moreover, denoting by
$a_1^2,\ldots,a_n^2$ such $n$ real positive eigenvalues and fixing $a_0 = 1$,
matrix $A_0$ in (\ref{eq:A0}) has $n$ eigenvalues in $s=0$ having $n$ distinct eigenvectors,
and $n+1$ pairs of purely imaginary conjugate eigenvalues in
$s = \pm \jmath \omega_0 a_k$, $k=0,\ldots, n$.
\end{lemma}

\begin{proof}
First, we establish that $I+BF$ has $n$ positive eigenvalues. 
This follows from the coordinate change 
$T = \sqrt{F}$ transforming the matrix into $I+\sqrt{F} B \sqrt{F}$,
which is a symmetric positive definite matrix because
$I$ is such and 
$B = \Sigma J_{n+1}^{-1}\Sigma^T$ is positive definite.
For later use, define $\bar{B}:=\sqrt{F}B\sqrt{F}$ 
and $N$, $\bar{N}$ by the relations $N^T N=(I+\bar{B})=\bar{N}^T \Lambda^2 \bar{N}$,
with $\Lambda=\diag(a_1,\ldots,a_n)$
and $\bar{N}^{-1}=\bar{N}^T$,
that is $\bar{N}$ is orthonormal 
(such factorizations are possible since $I+\bar{B}$ is positive definite).

\newcommand\blkdiag{{\rm blkdiag}}
\newcommand\imag{{\rm Im}}

In order to highlight the eigenstructure of $A_0$, 
a sequence of coordinate transformations will be used.
Consider a first change of coordinates $A_1=T_0 A_0 T_0^{-1}$ with
\begin{equation}\label{eq:T0}
T_0 = \blkdiag\left( -\sqrt{\frac{c}{\ell}F^{-1}}, \; \sqrt{\frac{c}{\ell}J_{n+1}}, \; \sqrt{J_{n+1}} \right),
\end{equation}
which yields (recalling that $\omega_0:=(\sqrt{\ell c})^{-1}$)
\begin{subequations}
\begin{align}
A_1&:= \omega_0\mymatrix{c|c}{0_{(2n+1)\times(2n+1)} & 
  \begin{array}{c} M \\[.2cm] I_{n+1}  \end{array} \\ \hline \\[-.3cm]
  \begin{array}{cc} -M^T & - I_{n+1} \end{array} &
  0_{(n+1)\times(n+1)}
    }, \label{eq:A1}\\
    M&:=\sqrt{F}\Sigma J_{-1} \sqrt{J_{n+1}^{-1}},\label{eq:M}
\end{align}
\end{subequations}
Note that since $\imag (\Sigma) = \RR^n$ and $\ker(\Sigma)=\imag(1_{n+1})$,
where $1_{n+1}$ is the vector in $\mathbb{R}^{n+1}$ having all components equal to 1,
it follows that
$$\imag(M) = \RR^n, \; \ker(M)=\imag(\nu), \; \nu:=\sqrt{J_{n+1}}J_{-1}1_{n+1},$$
and the matrix $[\nu \; M^T]\in\RR^{n+1}$ is invertible since by well known results
$\imag(M)\oplus \ker(M) = \RR^{n+1}$. 

Since $A_1$ in (\ref{eq:A1}) is real and skew symmetric, its
eigenvalues are either zero or in imaginary conjugate pairs.
To explicitly show them and the corresponding real invariant subspaces, 
consider matrix $V$ given by
\begin{equation}\label{eq:V}
V = \mymatrix{ccccc}{
		I_n			&0		&0		 &\bar{B} 	&0\\
		-M^T	 &\nu	&0		 &M^T 		&0\\
		0			&0		&\nu	&0 			 &M^T N^T}.
\end{equation}
Invertibility of $V$ can be seen by computing $V S_1 S_2$, with
\begin{align*}
S_1 = \smallmat{      I_n & 0 & 0 	& 0	& 0\\
			0 	& 1	& 0 	& 0 & 0\\
			0 	& 0 & 0 	& 1	& 0\\
			I_n	& 0 & I_n	& 0	& 0\\
			0 	& 0 & 0		& 0 & (N^T)^{-1}},\;
S_2 = \smallmat{
			(I_n+\bar{B})^{-1} 	& 0 	& -(I_n+\bar{B})^{-1}\bar{B} & 0 & 0\\
			0 & 1	  & 0 & 0 		& 0\\
			0 & 0 & I_n 	& 0	& 0\\
			0 & 0 	  & 0 & 1		& 0\\
			0 & 0 		& 0 & 0 		& I_n},
\end{align*}
yielding
\begin{align*}
V S_1 S_2 
		&=\smallmat{
			I_n+\bar{B}	&0  &\bar{B} &0 &0\\
			0 &\nu &M^T &0 &0\\
			0 &0 &0 &\nu &M^T}S_2
		=\smallmat{
			I_n	&0  &0 &0 &0\\
			0 &\nu &M^T &0 &0\\
			0 &0 &0 &\nu &M^T}.
\end{align*}
It is then possible to consider
the additional change of coordinates $A_2=T_1 A_1 T_1^{-1}$ where
$T_1=V^{-1}$. The computation of $A_2$ is immediate by 
expressing $A_1 V$ as $V A_2$, which yields
\begin{equation}\label{eq:A2}
A_2:= \omega_0\mymatrix{c|c|c}{
	0_{n\times n} & 0 &0\\ \hline
	0 & \begin{array}{cc} 0&1\\-1 &0\end{array} &0\\ \hline
	0 &0 &\begin{array}{cc} 0_{n\times n}&N^T\\-N &0_{n\times n}\end{array}}.
\end{equation}
Due to the block diagonal structure of $A_2$, 
it is clear that it has $n$ eigenvalues equal to $0$, 
a pair of imaginary eigenvalues at $\pm\jmath \omega_0$,
and the remaining eigenvalues equal to $\omega_0$ times 
the eigenvalues of 
\begin{align*}
\mymatrix{cc}{0_{n\times n}&N^T\\-N &0_{n\times n}}
&=\mymatrix{cc}{I_n&0\\0 &\bar{N}}
\mymatrix{cc}{0_{n\times n}&\Lambda\\-\Lambda &0_{n\times n}}
\mymatrix{cc}{I_n&0\\0 &\bar{N}^{-1}}\\
&=\mymatrix{cc}{I_n&0\\0 &\bar{N}}
\left(\mymatrix{cc}{0&1\\-1 &0}\otimes\Lambda\right)
\mymatrix{cc}{I_n&0\\0 &\bar{N}^{-1}},
\end{align*}
which are equal to $\pm\jmath a_i$, $i=1,\ldots,n$
since the eigenvalues of a Kronecker product of two matrices are given
by all the possible products between an eigenvalue of the first matrix 
(in this case, $\pm\jmath$) and an eigenvalue of the second matrix
(in this case, $a_i$, $i=1,\ldots,n$).
\end{proof}

\subsection{Proof of Theorem~\ref{th:design}}
\label{sec:proof}

Consider the circuit of Figure~\ref{fig:circuit} and its state space
equations with the state $x$ given in (\ref{eq:state}).  
In the following reasoning, we will use the coordinates $\tilde{x}=\tilde{T}x$ and 
$\hat{x}=T_0x=T_1^{-1}\tilde{x}$, where $\tilde{T}=T_1T_0$.

Considering the initial state as in the statement of Problem~\ref{prob_stat},
namely $x_0=v_{\circ}[0_n^T \; 1_n^T \; 0_{n+2}^T]^T$,
our aim is to show that the corresponding free response will yield
$x(T)=v_{\circ} [0_{2n}^T \; 1 \; 0_{n+1}^T]^T$.

In the $\hat{x}$ coordinates, the initial state $x_0$ becomes
$\hat{x}_0 = c_0[0_n^T \; 1_n^T \; 0_{n+2}^T]^T$
with $c_0=v_{\circ}\sqrt{\frac{c}{\ell}}$.
The corresponding expression of $\tilde{x}_0$ can be found
by the relation $\hat{x}_0=T_1^{-1}\tilde{x}_0$ with
$T_1^{-1}=V$ given by (\ref{eq:V}); partitioning 
$\tilde{x}_0$ according to the block columns of $V$,
and choosing 
$\delta_0:=(I_n+\bar{B})^{-1}\sqrt{F^{-1}}\delta$,
$\delta:=[1 \; 2 \; \ldots \; n]^T$, 
it follows that 
$
\tilde{x}_0
=\frac{c_0}{2}\smallmat{-\bar{B}\delta_0 \\ 1 \\ 
	0 \\  \delta_0 \\ 0}.
$ 
In fact, it is possible to verify that
\begin{align*}
V \tilde{x}_0 
	&=\mymatrix{ccccc}{
			I_n			&0		&0		 &\bar{B} 	&0\\
			-M^T	 &\nu	&0		 &M^T 		&0\\
			0			&0		&\nu	&0 			 &M^T N^T}
			\frac{c_0}{2}\smallmat{-\bar{B}\delta_0 \\ 1 \\ 0 \\  \delta_0 \\ 0} \\
	&=\frac{c_0}{2}\mymatrix{c}{0\\ M^T\sqrt{F^{-1}}\delta +\nu	\\ 0}=\mymatrix{c}{1_n\\ 0}=\hat{x}_0,
\end{align*}
by using the following relation:
\begin{align*}
&	M^T\sqrt{F^{-1}}\delta +\nu 
	=J_{-1}\sqrt{J_{n+1}^{-1}}\left(\Sigma^T \delta + J_{n+1} 1_{n+1}\right)\\
& \qquad \qquad =J_{-1}\sqrt{J_{n+1}^{-1}}\left(\mymatrix{c}{1_n\\ -n} + \mymatrix{c}{1_n\\ n}\right)
	=2\mymatrix{c}{1_n\\ 0}.
\end{align*}
Decompose now $\tilde{x}_0$ as $\tilde{x}_0=\tilde{x}_{01}+\tilde{x}_{02}+\tilde{x}_{03}$
with 
$$
\tilde{x}_{01} = \frac{c_0}{2}\mymatrix{c}{-\bar{B}\delta_0 \\ 0_{2n+2}}, 
\tilde{x}_{02} = \frac{c_0}{2}\smallmat{0_{n} \\ 1 \\ 0 \\ 0_{n+1}}, 
\tilde{x}_{03} = \frac{c_0}{2}\smallmat{0_{n+2} \\  \delta_0 \\ 0}, 
$$
and consider that, according to the structure in (\ref{eq:A2}),
$\tilde{x}_{01}$ only excites constant modes, 
$\tilde{x}_{02}$ only excites modes at frequency $\omega_0$
(which have a phase change between $t=0$ and $t=T$ of exactly
$\omega_0 T=\pi$), 
and $\tilde{x}_{03}$ only excites modes at frequency $\alpha_i\omega_0$,
$i=1,\ldots,n$ with $\alpha_i$ even
(which have a phase change between $t=0$ and $t=T$ of exactly
$\alpha_i\omega_0 T=2h\pi$, with $h\in\NN$).
It follows that 
$$
\tilde{x}(T) = \tilde{x}_{01} +\tilde{x}_{03} - \tilde{x}_{02}, 
$$
and then 
\begin{align*}
&\hat{x}(T) = V(\tilde{x}_{01} +\tilde{x}_{03} - \tilde{x}_{02}) 
= \frac{c_0}{2}\mymatrix{c}{0_n\\ M^T\sqrt{F^{-1}}\delta -\nu	\\ 0_{n+1}},\\
&\frac{c_0}{2}(M^T\sqrt{F^{-1}}\delta -\nu)
= \frac{c_0}{2}J_{-1}\sqrt{J_{n+1}^{-1}}\left(\mymatrix{c}{1_n\\ -n} - \mymatrix{c}{1_n\\ n}\right)\\
&\qquad=c_0\mymatrix{c}{0_n\\ \sqrt{n}}, 
\end{align*}
and finally, computing $x(T)=T_0^{-1}\hat{x}(T)$, 
the desired result $x(T)=v_{\circ} [0_{2n}^T \; 1 \; 0_{n+1}^T]^T$ is obtained,
which corresponds to having $v_{n+1}(T) = v_{\circ}$
and all other voltages and currents at zero,
which in turn implies $v_L(T) = nv_{n+1}(T) =n v_{\circ}$, as to be proven.

}

\section{Conclusion and perspectives}
\label{sec:conclusions}

%{\tt
%It would be insightful to study the applicability of existing
%upper and lower bounds on the number of real solutions of systems of
%polynomial equations, see \cite{Sot11} for a recent survey.
%}

We proved that
the design of an $n$-stage Marx generator electrical network 
exhibiting a desirable energy transfer
boils down to a structured pole assignment. This
 can be in turn formulated as a structured system of $n$ polynomial
equations in $n$ unknowns with rational coefficients. 
We have then illustrated a symbolic and a numerical approach to the
computation of its solutions. By extrapolating from the analyzed cases,
we conjecture that there is a finite number
of complex, hence real solutions to this polynomial system. We also conjecture that
there is at least one real solution. This motivates our
Assumptions~\ref{as:finitesol} and~\ref{as:onesol}.
The degrees of the polynomials are equal to
$1, 2, \ldots, n$ so that the B\'ezout bound on the number of complex
solutions, as well as the mixed volume of the support polytopes of
the polynomials \cite[Section 3]{stu02}, both equal $n!$. The number of computed real solutions
is however much less than this upper bound. It would be insightful to study
the applicability of existing
upper and lower bounds on the number of real solutions of systems of
polynomial equations, see \cite{Sot11} for a recent survey.

We solve the polynomial system of equations first with a symbolical method,
namely Gr\"obner bases and rational univariate representations. 
All real solutions are then obtained from the
real roots of a univariate polynomial of degree $n!$. This univariate
polynomial can be computed exactly (i.e. with rational coefficients)
and its real roots are isolated in rational intervals at any given
relative accuracy. Using state-of-the-art implementation of Gr\"obner
basis algorithms, we could solve the equations up to $n=6$ stages routinely
on a standard computer.  Then we solved the same system of polynomial equations
with a numerical method,
using Lasserre's hierarchy of convex LMI relaxations for polynomial optimization.
The advantage of this approach is that it is not necessary to represent
or enumerate all $n!$ complex solutions, and a particular real solution
optimal with respect to a given polynomial objective function can be found quite easily
up to $n=8$. The accuracy of the computed solutions may not be very good,
but this can be refined locally afterwards using, e.g. Newton's method.
Future work may involve the use of alternative solution methods
%We have not attempted to use other numerical methods for solving
%polynomial systems of equations, 
such as, e.g., 
homotopy or continuation
algorithms \cite{PHCpack,PHoM,Bates08}. 
%{\tt Since these algorithms track all complex
%$n!$ solutions, and real solutions are obtained afterwards by inspecting
%complex solutions with negligible imaginary part, we anticipate however that
%these methods will be also limited to $n=7$ or $n=8$.}
% However, for
%practical purposes, electrical networks of these dimensions
%seem to be largely sufficient.

\section*{Acknowledgments}        % Place acknowledgements
We are grateful to Mohab Safey El Din for technical input
on solving polynomial system of equations.
We also thank Chaouki T. Abdallah and Edl Schamiloglu for helpful
discussions.
Finally, we thank M. Francaviglia and A. Virz\`i for their preliminary
work on this subject.

\section *{Appendix}
\label{sec:maple}

In this appendix, we comment on some sample Maple and Matlab implementations of
the symbolic and numerical algorithms described in Sections~\ref{sec:symbolic}
and~\ref{sec:numerical}, respectively.

%\subsection{Practical resolution of the algebraic system}
\subsection*{Symbolic algorithm}

To solve the polynomial system of equations (\ref{eq:systemp})
we use special computing packages which can be called
directly from inside a computing sheet of the Maple computation software.

The first step computes a Gr\"obner basis in grevlex ordering.
This is achieved by means of the
{\tt fgb\_gbasis} procedure of the {\tt FGB} package,
see \cite{Fau1999,Fau2002}.
This implementation of Gr\"obner basis computation
is considered to be one of the most efficient available.
To speed up linear algebra computations,
the algorithm works in an integer ring with prime characteristic,
for a sufficient large prime number found iteratively.
Then, for finding the real solutions, a RUR is computed from the Gr\"obner basis,
see \cite{rou99}.
Finally,
for isolating the real roots of the univariate polynomial,
the procedure {\tt rs\_isolate\_gb} in the {\tt fgbrs} package is used.
It returns small rational intervals (as small as one wants) within
which the roots are guaranteed to be found.
See~\cite{RZ} for further references on the method used.

%\subsection{Example}

The whole process lies in a few lines of Maple code.
We first generate the polynomial system (corresponding to~\eqref{eq:systemp}),
denoted by {\tt p} in the following Maple sheet:

{\small
\begin{verbatim}
with(LinearAlgebra):with(PolynomialTools):
n:=4:B:=Matrix(n):
for i from 1 to n-1 do
 B(i,i):=2: B(i,i+1):=-1: B(i+1,i):=-1:
end do:
B(n,n):=(n+1)/n;
F:=Matrix(n,Vector(n,symbol=f),shape=diagonal):
d:=product(x-((2*j)^2-1),j=1..n):
p:=CoefficientList(collect(
   charpoly(B.F,x)-d,x),x);
\end{verbatim}
}

We then use the {\tt FGb} package to transform the system~{\tt p}  
in a new algebraic system referenced with the name {\tt GB}:

{\small
\begin{verbatim}
with(FGb):
fv:=[seq(f[i],i=1..n)]:
GB:=fgb_gbasis(p,0,fv,[]):
\end{verbatim}
}
Finally, the solutions are computed as follows:
{\small
\begin{verbatim}
with(fgbrs):
rs_isolate_gb(GB,fv);
\end{verbatim}
}

\begin{figure}[ht!]
\begin{center}
\includegraphics[width=0.8\columnwidth]{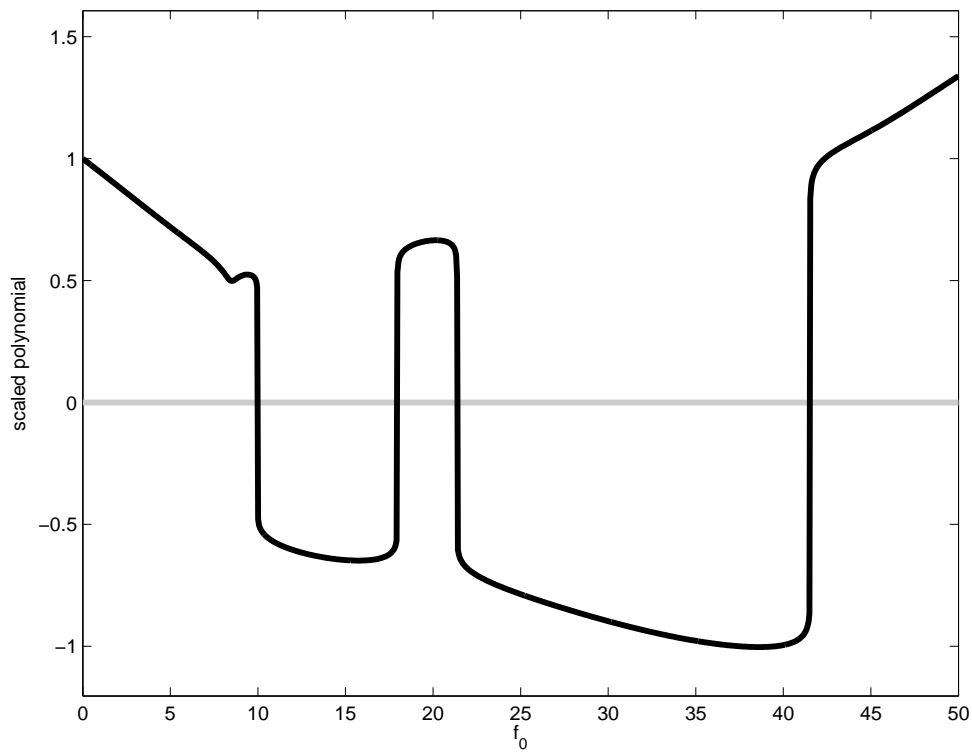}
\end{center}
\caption{Case $n=4$: rescaled graph of the univariate polynomial
whose 4 real roots parametrize the 4 real solutions of
the polynomial equations (\ref{eq:systemp}).}
\label{fig:racines}
\end{figure}

As an example, in the case $n=4$
%For instance, in the case $n=4$, we get the real solutions
%(truncated to 10 digits):
%{\tiny
%\[
%\begin{array}{l}
%(f_1 = 14.13305557, f_2 = 20.32225217, f_3 = 17.30660725, f_4 = 9.980936006)\\
%(f_1 = 15.07427225, f_2 = 7.606031707, f_3 = 24.09620357, f_4 = 17.95758795)\\
%(f_1 = 9.839448122, f_2 = 25.57984125, f_3 = 9.222883682, f_4 = 21.37252311)\\
%(f_1 = 9.352820256, f_2 = 12.34805352, f_3 = 10.34472793, f_4 = 41.52703727)
%\end{array}
%\]
%}
the univariate polynomial $r(f_{\circ})$ that enters the RUR (\ref{rur})
has degree 24, and it parametrizes all the $4!=24$ (complex)
solutions.
Out of 24, only 4 solutions are real.
 This polynomial, or rather its scaled version
\[
f_{\circ} \mapsto \mathrm{sign}(r(f_{\circ})) \,|r(f_{\circ})|^\frac{1}{24},
\]
which allows to more easily visualize its zero crossings,
is represented in Figure~\ref{fig:racines}.
It is computed as follows
{\small
\begin{verbatim}
with(Groebner):
R:=RationalUnivariateRepresentation(GB,x):
r:=R[1];
\end{verbatim}
}

In practice, the computation of a Gr\"obner basis becomes
very hard as $n$, the number of variables and equations, increases.
For the specific system of equations
(\ref{eq:systemp})
associated to the formulation (\ref{eq:pol_assign}) of the Marx generator design, 
we observe that for $2\le n \le 6$ the univariate polynomial
$r(f_{\circ})$ of the RUR has degree $n!$.
Moreover, the size of the (integer) coefficients of this polynomial
are very large. If one compares the degree of the polynomial,
720 in the case $n=6$, and the number of real roots, 12 in this case,
it is clear that the computation complexity is due to the very large
number of complex roots, which are of no interest for our engineering problem.

%On the other hand, one should notice that we try to assign eigenvalues
%of the form $(\alpha_i^2-1)^{-1}$, which are near zero when
%$n$ increases and $i$ near $n$. Hence, for large values of $n$, the diagonal matrix is near a non-
%invertible matrix. It explains the difficulty
%of the problem. \marginpar{ to discuss....} 

\subsection*{Numerical algorithm for $n=3$}

First we generate the polynomials $q_i(k)$, $i=1,\ldots,n$
corresponding to (\ref{eq:systemq}) and
defining the feasibility set $\mathscr K$ in problem (\ref{polyopt}). We use
the following Maple code, where we select $\alpha_i =2i$, $i=1,\ldots,
n$ just like in Section~\ref{sec:alg_sens}:

{\small
\begin{verbatim}
with(LinearAlgebra):with(PolynomialTools):
n:=3:B:=Matrix(n):for i from 1 to n-1 do
 B(i,i):=2: B(i,i+1):=-1: B(i+1,i):=-1:
end do: B(n,n):=(n+1)/n;
K:=Matrix(n,Vector(n,symbol=k),
          shape=diagonal):
p:=product(x-1/((2*j)^2-1),j=1..n):
q:=CoefficientList(collect(
    charpoly(MatrixInverse(B).F,x)-p,x),x);
\end{verbatim}
}

For $n=3$ this code generates the following polynomials
\[
\begin{array}{rcl}
q_1(k) & = & -\frac{3}{7}+\frac{5}{6}k_1+\frac{4}{3}k_2+\frac{3}{2}k_3 \\
q_2(k) & = & -\frac{53}{1575}+\frac{2}{3}k_1k_2+k_1k_3+k_2k_3 \\
q_3(k) & = & -\frac{1}{1575}+\frac{1}{2}k_1k_2k_3.
\end{array}
\]
These polynomials are then converted into Matlab format,
and we use the following GloptiPoly code for inputing problem (\ref{polyopt})
and solving the smallest possible LMI relaxation, i.e. $d=
\lceil\frac{3}{2}\rceil =2$ in problem
(\ref{lmirelax})
(note also the inequality constraints in {\tt P} corresponding to the
convexity requirement for the ``regular'' solution):

{\small
\begin{verbatim}
mpol k 3
K = [-3/7+5/6*k(1)+4/3*k(2)+3/2*k(3)
     -53/1575+2/3*k(1)*k(2)+k(1)*k(3)+k(2)*k(3)
     -1/1575+1/2*k(1)*k(2)*k(3)];
obj = 0;
for i = 1:length(k)
 for j = 1:length(k)
  obj = obj+(k(i)-k(j))^2;
 end
end
P = msdp(min(obj),K==0,k(1)-2*k(2)+k(3)>=0);
[stat,obj] = msol(P);
double(k)
\end{verbatim}
}

As pointed out in Remark~\ref{rem:relax}, already with the smallest
LMI relaxation $d=2$, 
we obtain a certificate of global optimality,
and a unique global minimizer (truncated at 5 digits):
$(k_1,\ k_2,\ k_3) = (9.3786\cdot10^{-2}, \; 8.6296\cdot10^{-2},\;
1.5690\cdot10^{-1})$ that corresponds to the second solution in the
second block of Table~\ref{tab:allnumbers} (the ``regular'' one as expected).
%We can check correctness of the solution with the following Matlab code
%which returns the values $35$, $15$ and $3$ up to 5 digits:
%
%{\small
%\begin{verbatim}
%n = length(k);
%B = 2*eye(n)-diag(ones(1,n-1),1)-...
%    diag(ones(1,n-1),-1); B(n,n) = (n+1)/n;
%sigma = eig(B*diag(1./double(k)))
%\end{verbatim}
%}

\bibliographystyle{plain}
\bibliography{refs}

\end{document}